\documentclass[10pt]{scrartcl}

\usepackage{tikz}
\usetikzlibrary{matrix,arrows,decorations.pathmorphing,shapes.geometric}

\usepackage{etex}
\usepackage{subcaption}
\usepackage[utf8]{inputenc}
\usepackage{a4wide}
\usepackage{graphicx}	
\usepackage{wrapfig}
\usepackage[hmarginratio={1:1},     
vmarginratio={1:1},     
textwidth=15cm,        
textheight=22cm,
heightrounded,]{geometry}

\usepackage{amsmath,accents}
\usepackage{amsthm}
\usepackage{amssymb}

\usepackage{mathrsfs, mathtools}
\usepackage{overpic}

\usepackage[utf8]{inputenc}

\usepackage{tikzit}

\tikzstyle{black dot}=[fill=black, draw=black, shape=circle, minimum size=3pt, inner sep=0pt]
\tikzstyle{black dot small}=[fill=black, draw=black, shape=circle, minimum size=3pt, inner sep=0pt]
\tikzstyle{big white circle}=[fill=white, draw=black, shape=circle, minimum width=0.75cm]
\tikzstyle{white dot big}=[fill=white, draw=black, shape=circle, inner sep=1pt]
\tikzstyle{white dot}=[fill=white, draw=black, shape=circle, minimum size=3pt, inner sep=0pt]
\tikzstyle{flat box}=[fill=white, draw=black, shape=rectangle, minimum width=2.5cm, minimum height=0.5cm]
\tikzstyle{square}=[fill=white, draw=black, shape=rectangle]
\tikzstyle{flat box 2}=[fill=white, draw=black, shape=rectangle, minimum height=0.5cm, minimum width=1.0cm]
\tikzstyle{over }=[front]
\tikzstyle{theta}=[fill=blue, draw=blue, shape=ellipse, minimum height=6pt, minimum width=6pt, inner sep=0pt]
\tikzstyle{thetabig}=[fill=blue, draw=blue, shape=ellipse, minimum width=1cm, minimum height=0.01cm]
\tikzstyle{thetainv}=[fill=white, draw=blue, shape=ellipse, minimum height=6pt, minimum width=6pt, inner sep=0pt]
\tikzstyle{thetabinv}=[fill=white, draw=blue, shape=ellipse, minimum width=1cm, minimum height=0.01cm]

\tikzstyle{mid arrow}=[-, postaction={on each segment={mid arrow}}]
\tikzstyle{end arrow}=[->]
\tikzstyle{red mid arrow}=[-, draw={rgb,255: red,214; green,42; blue,51}, postaction={on each segment={mid arrow}}, line width=1pt]
\tikzstyle{blue}=[-, draw=blue]
\tikzstyle{blue mid arrow}=[-, draw={rgb,255: red,23; green,37; blue,167}, postaction={on each segment={mid arrow}}, line width=1pt]
\tikzstyle{over}=[-, link]
\tikzstyle{mapsto}=[{|->}]

\usetikzlibrary{decorations.pathreplacing,decorations.markings}
\tikzset{
  on each segment/.style={
    decorate,
    decoration={
      show path construction,
      moveto code={},
      lineto code={
        \path [#1]
        (\tikzinputsegmentfirst) -- (\tikzinputsegmentlast);
      },
      curveto code={
        \path [#1] (\tikzinputsegmentfirst)
        .. controls
        (\tikzinputsegmentsupporta) and (\tikzinputsegmentsupportb)
        ..
        (\tikzinputsegmentlast);
      },
      closepath code={
        \path [#1]
        (\tikzinputsegmentfirst) -- (\tikzinputsegmentlast);
      },
    },
  },
  mid arrow/.style={postaction={decorate,decoration={
        markings,
        mark=at position .5 with {\arrow[#1]{stealth}}
      }}},
}
\tikzset{%
  link/.style    = { white, double = blue, line width = 1.8pt,
                     double distance = 0.4pt },
  channel/.style = { white, double = blue, line width = 0.8pt,
                     double distance = 0.6pt },
}

\usepackage[all,cmtip]{xy}

\makeatletter
\DeclareRobustCommand{\em}{%
	\@nomath\em \if b\expandafter\@car\f@series\@nil
	\normalfont \else \slshape \fi}
\makeatother

\usepackage[hidelinks]{hyperref}

\numberwithin{equation}{section}
\usepackage{mathtools}
\mathtoolsset{showonlyrefs}
\numberwithin{equation}{section}

\newtheoremstyle{style1}
{13pt}
{13pt}
{}
{}
{\normalfont\bfseries}
{.}
{.5em}
{}

\theoremstyle{style1}

\newtheorem{definition}{Definition}[section]
\newtheorem{example}[definition]{Example}
\newtheorem{remark}[definition]{Remark}

\makeatletter
\newtheorem*{repd@theorem}{\repd@title}
\newcommand{\newrepdtheorem}[2]{%
	\newenvironment{repd#1}[1]{%
		\def\repd@title{#2 \ref{##1}}%
		\begin{repd@theorem}}%
		{\end{repd@theorem}}}
\makeatother

\newrepdtheorem{definition}{Definition}

\newcommand{\catf}[1]{{\mathsf{#1}}}

\newtheoremstyle{style2}
{13pt}
{13pt}
{\slshape}
{}
{\normalfont\bfseries}
{.}
{.5em}
{}

\theoremstyle{style2}

\makeatletter
\newtheorem*{rep@theorem}{\rep@title}
\newcommand{\newreptheorem}[2]{%
	\newenvironment{rep#1}[1]{%
		\def\rep@title{#2 \ref{##1}}%
		\begin{rep@theorem}}%
		{\end{rep@theorem}}}
\makeatother

\newreptheorem{theorem}{Theorem}
\newreptheorem{corollary}{Corollary}

\newtheorem{lemma}[definition]{Lemma}
\newtheorem{theorem}[definition]{Theorem}
\newtheorem{proposition}[definition]{Proposition}
\newtheorem{corollary}[definition]{Corollary}

\def\hocolim{\mathrm{hocolim}}

\usepackage{tikz}
\usetikzlibrary{matrix,arrows,decorations.pathmorphing,shapes.geometric}
\usepackage{tikz-cd}

\usepackage{enumitem}

\usepackage{needspace}
\newcommand{\spaceplease}{\needspace{5\baselineskip}}

\newcommand{\Z}{\mathbb{Z}}

\newcommand{\Grpd}{\catf{Grpd}}
\newcommand{\Map}{\catf{Map}}
\newcommand{\Diff}{\catf{Diff}}
\newcommand{\Env}{\catf{U}\,}
\newcommand{\Envint}{\catf{U}_{\!\int}\,}
\newcommand{\Envintg}{\catf{U}_{\!\int}^g\,}
\newcommand{\Envintone}{\catf{U}_{\!\int}^1\,}

\newcommand{\ra}[1]{\xrightarrow{\ #1 \ }}

\newcommand{\algo}{\mathfrak{A}_\cat{O}}
\let\DH\undefined
\newcommand{\DH}{\catf{DH}}

\newcommand{\framed}{\catf{f}E_2}

\newcommand{\Surf}{\catf{Surf}}

\newcommand{\hocolimsub}[1]{\underset{#1}{\operatorname{hocolim}}\,}

\newcommand{\Hbdy}{\catf{Hbdy}}

\usepackage{pifont}

\newcommand{\As}{\catf{As}}

\newcommand{\Top}{\catf{Top}}
\newcommand{\Legs}{\catf{Legs}}

\newcommand{\ModAlg}{\catf{ModAlg}\,}
\newcommand{\CycAlg}{\catf{CycAlg}}
\newcommand{\cat}[1]{\mathcal{#1}}

\newcommand{\Aut}{\operatorname{Aut}}

\newcommand{\End}{\operatorname{End}}

\newcommand{\Hom}{\operatorname{Hom}}

\newcommand{\id}{\operatorname{id}}

\newcommand{\Cat}{\catf{Cat}}

\newcommand{\DS}{\text{/\hspace{-0.1cm}/}}

\let\to\undefined
\newcommand{\to}{\longrightarrow}
\let\mapsto\undefined
\newcommand{\mapsto}{\longmapsto}

\newcommand{\Rexf}{\catf{Rex}^{\mathsf{f}}}
\newcommand{\Lexf}{\catf{Lex}^\mathsf{f}}

\newcommand{\Vect}{\catf{Vect}}
\newcommand{\fVect}{\catf{Vect}^\catf{f}}

\newcommand{\Forests}{\catf{Forests}}
\newcommand{\Graphs}{\catf{Graphs}}

\newcommand{\opp}{\text{opp}}

\let\colon\undefined\newcommand{\colon}{:}

\DeclareMathSymbol{\Phiit}{\mathalpha}{letters}{"08} 
\DeclareMathSymbol{\Psiit}{\mathalpha}{letters}{"09}
\DeclareMathSymbol{\Sigmait}{\mathalpha}{letters}{"06}
\DeclareMathSymbol{\Xiit}{\mathalpha}{letters}{"04}
\DeclareMathSymbol{\Piit}{\mathalpha}{letters}{"05}\let\Pi\undefined\newcommand{\Pi}{\Piit}
\DeclareMathSymbol{\Gammait}{\mathalpha}{letters}{"00}
\DeclareMathSymbol{\Omegait}{\mathalpha}{letters}{"0A}\let\Omega\undefined\newcommand{\Omega}{\Omegait}
\DeclareMathSymbol{\Upsilonit}{\mathalpha}{letters}{"07}
\DeclareMathSymbol{\Thetait}{\mathalpha}{letters}{"02}
\DeclareMathSymbol{\Lambdait}{\mathalpha}{letters}{"03}\let\Lambda\undefined\newcommand{\Lambda}{\Lambdait}
\newcommand{\envframed}{\mathbb{L}\mathsf{U}\,\framed}
\newcommand{\envframedgn}{\mathbb{L}\mathsf{U}^{g,n}\,\framed}
\newcommand{\envframedt}{\mathbb{L}\mathsf{U}^{1,0}\,\framed}
\let\Phi\undefined\newcommand{\Phi}{\Phiit}
\let\Sigma\undefined\newcommand{\Sigma}{\Sigmait}
\let\Psi\undefined\newcommand{\Psi}{\Psiit}
\let\Gamma\undefined\newcommand{\Gamma}{\Gammait}

\usepackage{multicol}
\newenvironment{pnum}{\begin{enumerate}[label=(\roman*)]}{\end{enumerate}}

\setkomafont{section}{\rmfamily\large}
\setkomafont{subsection}{\rmfamily}
\setkomafont{subsubsection}{\rmfamily}
\setkomafont{subparagraph}{\rmfamily} 

\makeatletter
\renewcommand\section{\@startsection {section}{1}{\z@}%
	{-3.5ex \@plus -1ex \@minus -.2ex}%
	{2.3ex \@plus.2ex}%
	{\normalfont\scshape\centering}}
\makeatother

\usepackage{titletoc}

\dottedcontents{section}[1em]{\rmfamily}{1em}{0.2cm}
\dottedcontents{subsection}[0em]{}{3.3em}{1pc}

\begin{document}

	\vspace*{-1.5cm}	\begin{flushright}\small		{\sffamily CPH-GEOTOP-DNRF151} 	\end{flushright}	\vspace{5mm}	\begin{center}	\textbf{\Large{Classification of Consistent Systems of \\[0.5ex]
				Handlebody Group Representations}}\\	\vspace{1cm}	{\large Lukas Müller $^{a}$} \ and \ \ {\large Lukas Woike $^{b}$}\\ 	\vspace{5mm}{\slshape $^a$ Max-Planck-Institut f\"ur Mathematik\\ Vivatsgasse 7 \\  53111 Bonn, Germany}\\ \emph{lmueller4@mpim-bonn.mpg.de }	\\[7pt]	{\slshape $^b$ Institut for Matematiske Fag\\ K\o benhavns Universitet\\	Universitetsparken 5 \\  2100 K\o benhavn \O , Denmark }\\ \ \emph{ljw@math.ku.dk }\end{center}	\vspace{0.3cm}	
	\begin{abstract}\noindent 
		The classifying spaces of handlebody groups form a modular operad. Algebras over the handlebody operad yield systems of representations of handlebody groups that are compatible with gluing. We prove that algebras over the modular operad of handlebodies with values in an arbitrary symmetric monoidal bicategory $\mathcal{M}$ (we introduce for these the name \emph{ansular functor}) are equivalent to self-dual balanced braided algebras in $\mathcal{M}$. After specialization to a \emph{linear} framework, this proves that consistent systems of handlebody group representations on finite-dimensional vector spaces are equivalent to ribbon Grothendieck-Verdier categories in the sense of Boyarchenko-Drinfeld. Additionally, it produces a concrete formula for the vector space assigned to an arbitrary handlebody in terms of a generalization of Lyubashenko's coend. Our main result can be used to obtain an ansular functor from vertex operator algebras subject to mild finiteness conditions. This includes examples of vertex operator algebras whose representation category has a non-exact monoidal product. 	\end{abstract}

\tableofcontents

\spaceplease
\section{Introduction and summary}
In this article, we 
classify systems of 
handlebody group representations on finite-dimensional vector 
spaces over an algebraically closed field $k$, subject to the requirement that these representations are compatible with the gluing of handlebodies. 
We call such a consistent system of handlebody group representations an \emph{ansular functor} and will give a precise definition momentarily. 
The word \emph{ansular functor} is derived from the Latin word \emph{ansa} for handle. It is supposed to emphasize the close relationship to the notion of a \emph{modular functor} \cite{Segal,ms89,turaev,tillmann,baki} that
plays a key role at the cross-roads of low-dimensional topology, representation theory and mathematical physics, and is defined, roughly, as a 
 system of representations of mapping class groups	 of \emph{surfaces} that is compatible with gluing.
  Every modular functor yields an ansular functor by restriction along the boundary functor from handlebodies to surfaces (Remark~\ref{remmodular}).

 Systems of handlebody group representations are easier to control than systems of mapping class group representations. As a consequence, the notion of an ansular functor is way more flexible, and much stronger statements can be made. 
 This somewhat pragmatic reason explains partly
 the focus on ansular functors in this paper, but there is also a deeper algebraic reason:
  The classification of ansular functors can be accomplished in terms of structures of	 independent \emph{algebraic} interest, namely monoidal categories with Grothendieck-Verdier duality as defined by Boyarchenko and Drinfeld \cite{bd} that can be built from several rich sources in quantum algebra. In contrast to finite tensor categories in the sense of Etingof and Ostrik \cite{etingofostrik}, Grothendieck-Verdier categories are not necessarily rigid and can therefore have a non-exact monoidal product.
 Providing a topological perspective on Grothendieck-Verdier duality is an important goal in itself.

 Before explaining Grothendieck-Verdier duality in more detail, let us expand on the notion of an ansular functor:
 As mentioned above, an ansular functor assigns to a handlebody $H$ with $n$ disks embedded in $\partial H$, that
 each carry a label $X_j$ for $1\le j\le n$ chosen from a fixed label set,
 a finite-dimensional vector space $B_H(X_1,\dots,X_n)$ 
 with an action of the mapping class group $\Map(H)$ of the handlebody $H$, see Figure~\ref{fighandlebody}.
 In this paper, all diffeomorphisms of handlebodies (or surfaces) and their mapping classes will be \emph{orientation-preserving}.
 The vector spaces $B_H(X_1,\dots,X_n)$, that sometimes are referred to as \emph{spaces of conformal blocks},
 are subject to various consistency conditions 
 and, in particular, a gluing axiom that we refer to as \emph{excision}. This will entail that the label \emph{set} actually inherits the structure of a \emph{linear category} by evaluation of the ansular functor on the cylinder.

 	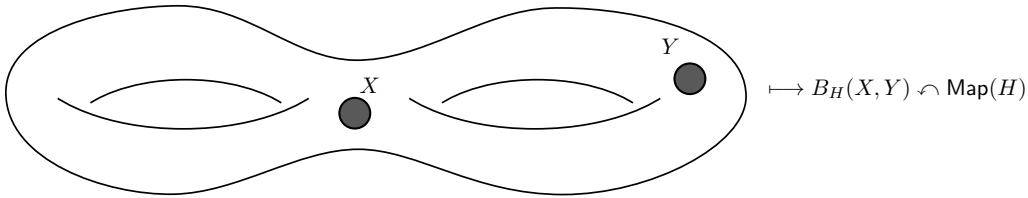
\begin{figure}[h]
 	\centering
 	\scalebox{.85}{
\begingroup%
  \makeatletter%
  \providecommand\color[2][]{%
    \errmessage{(Inkscape) Color is used for the text in Inkscape, but the package 'color.sty' is not loaded}%
    \renewcommand\color[2][]{}%
  }%
  \providecommand\transparent[1]{%
    \errmessage{(Inkscape) Transparency is used (non-zero) for the text in Inkscape, but the package 'transparent.sty' is not loaded}%
    \renewcommand\transparent[1]{}%
  }%
  \providecommand\rotatebox[2]{#2}%
  \newcommand*\fsize{\dimexpr\f@size pt\relax}%
  \newcommand*\lineheight[1]{\fontsize{\fsize}{#1\fsize}\selectfont}%
  \ifx\svgwidth\undefined%
    \setlength{\unitlength}{529.14479589bp}%
    \ifx\svgscale\undefined%
      \relax%
    \else%
      \setlength{\unitlength}{\unitlength * \real{\svgscale}}%
    \fi%
  \else%
    \setlength{\unitlength}{\svgwidth}%
  \fi%
  \global\let\svgwidth\undefined%
  \global\let\svgscale\undefined%
  \makeatother%
  \begin{picture}(1,0.15913682)%
    \lineheight{1}%
    \setlength\tabcolsep{0pt}%
    \put(0,0){\includegraphics[width=\unitlength,page=1]{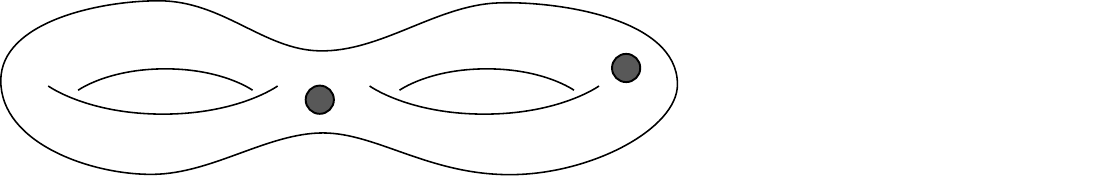}}%
    \put(0.29385602,0.08567282){\color[rgb]{0,0,0}\makebox(0,0)[lt]{\lineheight{1.25}\smash{\begin{tabular}[t]{l}$X$\end{tabular}}}}%
    \put(0.54449153,0.11802717){\color[rgb]{0,0,0}\makebox(0,0)[lt]{\lineheight{1.25}\smash{\begin{tabular}[t]{l}$Y$\end{tabular}}}}%
    \put(0.63292065,0.08426005){\color[rgb]{0,0,0}\makebox(0,0)[lt]{\lineheight{1.25}\smash{\begin{tabular}[t]{l}$\longmapsto B_H(X,Y) \curvearrowleft \Map(H)$\end{tabular}}}}%
  \end{picture}%
\endgroup%
}
 	\caption{A handlebody $H$ with two disks embedded in $\partial H$ labeled with $X$ and $Y$. }
 	\label{fighandlebody}
 \end{figure}

In this na\" ive description, it is a little inconvenient to make the notion of an ansular functor precise, especially the behavior under gluing. Moreover, the focus on \emph{linear} representations is artificial from a topological perspective. The theory of cyclic and 
modular operads of Getzler and Kapranov \cite{gk,gkmod} allows us to give a more compact description. A \emph{cyclic operad} is an operad which comes with a way to cyclically permute 
 the inputs with the output. A \emph{modular operad} additionally admits self-compositions of operations.
For any cyclic or modular operad with values in the bicategory of categories, one can define cyclic or modular algebras with values in any symmetric monoidal bicategory. 
 We recall the definition of cyclic and modular operads and algebras over them in Section~\ref{secoperads} using the very elegant approach given by Costello~\cite{costello} based on graph categories. 
The most important modular operad in this paper is the (groupoid-valued) modular operad $\Hbdy$ of handlebodies and their mapping classes.
It allows us to define the notion of an ansular functor:

\begin{repddefinition}{defansularfunctordef}
	An \emph{ansular functor}
	with values in a symmetric monoidal bicategory $\cat{M}$
	is an $\cat{M}$-valued modular algebra over the handlebody operad. 
\end{repddefinition}

After an unpacking, this definition tells us that an ansular functor 
has an underlying object $A\in \cat{M}$ and for any $n\ge 0$ functors $\Hbdy(n)\to \cat{M}(A^{\otimes n},I)$, where $\Hbdy(n)$ is the arity $n$ term of the category-valued handlebody operad and $\cat{M}(A^{\otimes n},I)$ is the morphism category from the $n$-fold monoidal product of $A$ to the monoidal unit $I$ in $\cat{M}$. These functors are subject to equivariance and composition axioms that hold up to coherent isomorphism (but all of that is now hard-coded in the notion of a modular algebra). 
In particular, for a handlebody $H$ with $n$ disks embedded in $\partial H$, we obtain a 1-morphism $A^{\otimes n}\to I$, and elements in the handlebody group of $H$ give rise to 2-automorphisms of this 1-morphism. 

In order to make contact with the na\" ive description of an ansular functor given above, one specializes 
$\cat{M}$ to be $\Lexf$, the symmetric monoidal bicategory of \begin{itemize}
	\item finite categories, i.e.\
$k$-linear abelian categories 
which have finite-dimensional morphism spaces, enough projective
objects, finitely many 
simple objects up to isomorphisms; moreover, every object has finite length, \item 
left exact functors as 1-morphisms, \item natural transformations as 2-morphisms.\end{itemize} The monoidal product is the Deligne product $\boxtimes$. The monoidal unit is the category $\fVect$ of finite-dimensional vector spaces over $k$. The field $k$ is fixed throughout and therefore suppressed in the notation. It is assumed to be algebraically closed.
If $A$ is an ansular functor with values in $\Lexf$, then the underlying category $\cat{C}$
of $A$ takes the role of a category of labels, and the evaluation
of $\cat{C}^{\boxtimes n} \to \fVect$ on a family of labels $(X_1,\dots,X_n)$ is a vector space $B_H(X_1,\dots,X_n)$  ---
just as in the above na\" ive description. It carries an action of the handlebody group of $H$. However, let us emphasize again that, topologically speaking, there is no reason to specialize to $\cat{M}=\Lexf$, and we will therefore formulate our results mostly for a general symmetric monoidal bicategory $\cat{M}$. Nonetheless, the case $\cat{M}=\Lexf$ is extremely important because it allows us to exhibit interesting classes of examples.

Let us now state the classification result for ansular functors. To this end, we need the following definition:

\begin{repddefinition}{defselfdual}
A \emph{self-dual balanced braided algebra} in a symmetric monoidal bicategory $\cat{M}$ is an object $A \in \cat{M}$
equipped with the following structure: \begin{itemize}
	\item $A$ is a \emph{balanced braided algebra in $\cat{M}$}, i.e.\ $A$ is equipped with 
	\begin{itemize}
		\item 
a multiplication $\mu : A \otimes A \to A$ which is associative and unital
	(the unit is a 1-morphism $u:I\to A$ from the unit $I$ of $\cat{M}$ to $A$)
	up to coherent isomorphism,
	\item an isomorphism $c: \mu \to \mu^\opp = \mu \circ \tau$ (here $\tau$ is the symmetric braiding of $\cat{M}$) called \emph{braiding} subject to the usual hexagon relations (also known as Yang-Baxter equations), 
	\item an isomorphism $\theta : \id_A \to \id_A$ called \emph{balancing} subject to the requirements
	\begin{align}
		\theta \circ \mu  &= c^2 \circ \mu(\theta  \otimes \theta) \ , \\
		\theta \circ u &= \id_u \ . 
	\end{align}
\end{itemize}
\item $A$ is equipped with a non-degenerate pairing 
$\kappa : A \otimes A \to I$, i.e.\ a 1-morphism which exhibits $A$ as its own dual in the homotopy category of $\cat{M}$, and an isomorphism $\gamma : \kappa(u\otimes \mu)\to \kappa$ such that  \begin{itemize}\item the isomorphism
	$\kappa(u\otimes \mu(u,-))\longrightarrow \kappa(u,-)$ induced by 
	$\gamma$ agrees with the isomorphism induced by the unit constraint,
	\item $\kappa(\theta \otimes \id_A)=\kappa(\id_A \otimes \theta)$.
	\end{itemize}
\end{itemize}
\end{repddefinition}

\spaceplease
\begin{reptheorem}{thmansular2}[Classification of ansular functors]
	For any symmetric monoidal bicategory $\cat{M}$, there is an equivalence between
	$\cat{M}$-valued ansular functors 
	and self-dual
	balanced braided algebras in $\cat{M}$.
\end{reptheorem}
The result can be understood as an equivalence of suitable \emph{2-groupoids} of ansular functors and self-dual balanced braided algebras.
Theorem~\ref{thmansular2} is a classification in the sense that it describes the algebraic structure `ansular functor' explicitly through a finite list of generators and relations.
Needless to say, it does not allow us to `list' all possible ansular functors (as we will explain below, this algebraic structure is at least as rich as ribbon categories or vertex operator algebras up to equivalence of their representation category).

\subparagraph{The linear version of the main result.}
As mentioned above,
one of  the most important cases for applications of Theorem~\ref{thmansular2} in quantum algebra is $\cat{M}=\Lexf$.
In \cite{cyclic} it is proven that $\Lexf$-valued
self-dual balanced braided algebras are equivalent 
to ribbon Grothendieck-Verdier categories 
in the sense of Boyarchenko and Drinfeld \cite{bd}.
Roughly, a Grothendieck-Verdier category is a monoidal category $\cat{C}$ with a distinguished object $K\in\cat{C}$ such that for any $X\in\cat{C}$ the functor $\cat{C}(K,X\otimes-)$ is representable (we call the representing object $DX\in\cat{C}$) and such that the functor $D:	\cat{C}\to\cat{C}^\opp,X\mapsto DX$ is an equivalence.
A ribbon Grothendieck-Verdier category is additionally equipped with a braiding $c_{X,Y} :X\otimes Y \to Y\otimes X$ and a balancing $\theta_X : X \to X$ in a compatible way; we give the details in Section~\ref{secansular}. 
The notion of Grothendieck-Verdier duality is based on 
Barr's notion of a \emph{$\star$-autonomous category} \cite{barr}
and should be understood as a weak form of rigidity (existence of duals in the category)
which, in 
contrast to the usual rigidity, 
does not imply the exactness of the monoidal product.
This gives us the following linear version of the classification result: 

\begin{reptheorem}{thmclassansfunctor}[Classification of ansular functors --- linear version]
	There is an equivalence between
	ansular functors with values in  $\Lexf$ 
	and ribbon Grothendieck-Verdier categories in $\Lexf$.
\end{reptheorem}

In Corollary~\ref{corallansular}, we give a fully explicit description of all $\Lexf$-valued ansular functors in terms of their genus zero part and a generalization of the Lyubashenko coend \cite{lyu,lyulex}.

Any finite ribbon category in the sense of \cite{egno} is a ribbon Grothendieck-Verdier category in $\Lexf$. Therefore, an example is in particular the category of finite-dimensional modules over a finite-dimensional ribbon Hopf algebra, see Example~\ref{exhopf} for a brief comment.
Most remarkably, by a recent of result of Allen, Lentner, Schweigert and Wood \cite{alsw}
building on the tensor product theory of Huang, Lepowsky and Zhang \cite{hlz},
suitable categories of modules over a vertex operator algebra
form a ribbon Grothendieck-Verdier category with the contragredient representation as the (not necessarily rigid!) dual.
This includes e.g.\ the $\cat{W}_{2,3}$ triplet model \cite{grw} which is known to have a non-exact monoidal product.
Using these algebraic results, Theorem~\ref{thmansular} gives us, for any vertex operator algebra subject to rather mild conditions, 
an ansular functor and hence spaces of
conformal blocks 
which, on the level of the representation category, are explicitly computable and 
carry at the very least
handlebody group representations (Corollary~\ref{corvoa}).
One should appreciate that both the algebraic results of \cite{alsw} and the topological ones of \cite{cyclic} and the present paper go beyond the usual framework of finite tensor categories \cite{etingofostrik,egno} with built-in rigidity.
In particular, we see:

\spaceplease
\begin{repcorollary}{cornonexact}
	There exist $\Lexf$-valued ansular functors whose underlying monoidal category is not exact and hence not rigid.
\end{repcorollary}

An important subclass of ribbon Grothendieck-Verdier categories is formed by 
those whose dualizing object is the monoidal unit
(this is closely related to the notion of an
\emph{$r$-category}~\cite[Definition~0.5]{bd}). We give --- based on Theorem~\ref{thmclassansfunctor} --- a necessary and sufficient \emph{topological}
criterion when this is the case (Corollary~\ref{corrcat}).

\subparagraph{The strategy for the proof the main result.}
Finally, in this last section of the introduction, we would like to give the idea for the proof of the main result.
We will also use this opportunity to explain  
the relation to our previous articles \cite{cyclic,mwdiff} in which we have initiated the study of the relation between Grothendieck-Verdier duality and the geometry of handlebodies. 
The present article, or more specifically the classification result for ansular functors, concludes the study of this relation.
Moreover, we would like to highlight several results of independent interest
that appear on the path towards the main result, in particular a connection
between the dihedral homology, as defined by Loday in \cite{loday}, of the topological algebra of unary operations of 
a cyclic operad and its genus one contribution of the \emph{modular envelope} defined by Costello \cite{costello}.
This relation is exploited for the framed $E_2$-operad. 

Let us give a brief overview:
Since the framed $E_2$-operad 
can be identified with the operad of oriented 
genus zero surfaces, it	 comes 
with  the structure of a cyclic operad.
By means of Costello's modular envelope \cite{costello},
any cyclic operad can be freely completed to a modular operad. It is understood here that this is done in a homotopically correct, i.e.\ derived way. 
Giansiracusa \cite{giansiracusa} proves that there is a canonical map from the derived modular envelope $\envframed$ of framed $E_2$ (as topological operad) to the modular operad $\Hbdy$ of handlebodies, i.e.\ the modular operad built from classifying spaces of handlebody groups. 
Giansiracusa proves that this map
\begin{align}
	\envframed \to \Hbdy
\end{align}
induces a bijection
on path components
and a homotopy equivalence on all path components 
except the one 	for the solid closed torus
(moreover, the three-dimensional ball and the three-dimensional ball with one embedded disk are excluded here).
Explicitly, this means that we obtain in total arity $n$ (meaning that the number of inputs plus output is $n$) the decomposition
\begin{align}
	\envframed (n) = \bigsqcup_{g\ge 0} \envframedgn \ , \label{eqndecomposition}
\end{align}
into connected components $\envframedgn$. For $(g,n)\neq (1,0)$, 
the space $\envframedgn$ is a classifying space for the mapping class group $\Map(H_{g,n})$ of the handlebody $H_{g,n}$ of genus $g$ and $n$ disks embedded in the boundary
(note that if $(g,n)\neq (0,0),(1,0)$, the mapping class group is equivalent to the topological group of diffeomorphisms).
As one of the first  technical steps, we need to understand the genus one component $	\envframedt $.
To this end, the following result valid for \emph{any} groupoid-valued cyclic operad is key:
\begin{reptheorem}{thmgenusonegeneral}
	For any groupoid-valued cyclic operad $\cat{O}$, the genus one contribution to the derived modular envelope of $\cat{O}$ is homotopy equivalent to the dihedral homology of the topological algebra of unary operations of $\cat{O}$.
\end{reptheorem}
Together with the results of \cite{mwdiff}, this will tell us that the genus one component $\envframedt$
	of the derived modular envelope of the topological framed $E_2$-operad $\framed$ is
	 homotopy equivalent  to $B \Diff (H_{1,0})$, where $\Diff (H_{1,0})$ is the topological group of diffeomorphisms of the genus one handlebody $H_{1,0}=\mathbb{S}^1\times\mathbb{D}^2$; in formulae,
	\begin{align}
		\envframedt   \simeq   B \Diff (H_{1,0}) \ \ .
	\end{align}
When combined with the results of \cite{giansiracusa}  (there are some additional subtleties coming from the fact that our treatment of unary operations is different from the one in \cite{giansiracusa}), we may then conclude that, after taking fundamental groupoids,
the canonical map from the derived modular envelope of $\framed$ to $\Hbdy$ yields an equivalence of
groupoid-valued modular operads (Theorem~\ref{thmmodenvpi}).

Additionally,
we prove the following modular extension result for cyclic algebras formulated in terms of a `category-valued version' $\Envint\cat{O}$ of the modular envelope built using the Grothendieck construction (see Section~\ref{modenvsec} for details): 
\begin{reptheorem}{thmositionmodext}
	Let $\cat{O}$ be a groupoid-valued cyclic operad and $\cat{M}$ a symmetric monoidal bicategory.
The modular extension of cyclic algebras and the  restriction afford mutually weakly inverse equivalences between the 2-groupoids of
cyclic $\cat{O}$-algebras in $\cat{M}$
and  modular $\Pi |B \Envint\cat{O}|$-algebras in $\cat{M}$, i.e.\
\begin{equation}
\begin{tikzcd}
\CycAlg(\cat{O}) \ar[rrrr, shift left=2,"\text{modular extension}"] &&\simeq&& \ar[llll, shift left=2,"\text{restriction}"] \ModAlg( \Pi |B \Envint\cat{O}|) \ . 
\end{tikzcd}
\end{equation}
\end{reptheorem}
Here $\Pi$ denotes the fundamental groupoid, $B$ the nerve and $|-|$ the geometric realization.

We then conclude that ansular functors in a symmetric monoidal bicategory are equivalent to cyclic framed $E_2$-algebras (Theorem~\ref{thmansular}). This reduces the classification of ansular functors to the classification of cyclic framed $E_2$-algebras in a symmetric monoidal bicategory which was given in~\cite{cyclic}.

	\vspace*{0.2cm}\textsc{Acknowledgments.} We are grateful to Adrien Brochier, 	Maxime Ramzi,	
	Christoph Schweigert,
Nathalie Wahl and Simon Wood
for  helpful discussions related to this project.
We also thank the anonymous referees whose comments have greatly helped improve the manuscript.
LM gratefully acknowledges support by the Max Planck Institute for Mathematics in Bonn.
LW gratefully acknowledges support by 
the Danish National Research Foundation through the Copenhagen Centre for Geometry
and Topology (DNRF151)
and  by the European Research Council (ERC) under the European Union's Horizon 2020 research and innovation programme (grant agreement No.~772960).

\section{Ansular functors\label{secoperads}}
The purpose of this section is to recall the necessary definitions from the theory of cyclic and modular operads \cite{gk,gkmod} and their algebras to arrive  at a concise definition of the notion of an ansular functor.
We start by briefly summarizing Costello's description of cyclic and modular operads \cite{costello}:
We denote by $\Graphs$ the category whose objects are finite (possibly empty) disjoint unions of corollas (a corolla is a contractible graph with one vertex and a finite number $n\ge 0$ of legs attached to it). We denote the corolla with $n+1$ legs by $T_n$.
The morphisms of $\Graphs$ are defined as follows:
For a graph $\Gamma$, denote by $\nu(\Gamma)$ the graph obtained by cutting open all internal edges (this is a disjoint union of corollas) and by $\pi_0(\Gamma)$ the graph obtained by contracting all internal edges of $\Gamma$ (again, this is a disjoint union of corollas), see Figure~\ref{nupicap}.
Now for two disjoint unions of corollas $T$ and $T'$, a morphism $T\to T'$ in $\Graphs$ is an equivalence class of graphs $\Gamma$ together with identifications $\nu(\Gamma)\cong T$ and $\pi_0(\Gamma)\cong T'$ up to isomorphisms of graphs preserving the identifications. Therefore, cyclic symmetries are a special type of morphisms in $\Graphs$; in fact, $\Aut_\Graphs(T_n)$ is the permutation group of the set of legs, i.e.\ the permutation group on $n+1$ letters. We refer to \cite[Section~2]{cyclic} for a  more detailed presentation.

	\begin{figure}[h]
	\centering
	\scalebox{.99}{
\begingroup%
  \makeatletter%
  \providecommand\color[2][]{%
    \errmessage{(Inkscape) Color is used for the text in Inkscape, but the package 'color.sty' is not loaded}%
    \renewcommand\color[2][]{}%
  }%
  \providecommand\transparent[1]{%
    \errmessage{(Inkscape) Transparency is used (non-zero) for the text in Inkscape, but the package 'transparent.sty' is not loaded}%
    \renewcommand\transparent[1]{}%
  }%
  \providecommand\rotatebox[2]{#2}%
  \newcommand*\fsize{\dimexpr\f@size pt\relax}%
  \newcommand*\lineheight[1]{\fontsize{\fsize}{#1\fsize}\selectfont}%
  \ifx\svgwidth\undefined%
    \setlength{\unitlength}{248.22258297bp}%
    \ifx\svgscale\undefined%
      \relax%
    \else%
      \setlength{\unitlength}{\unitlength * \real{\svgscale}}%
    \fi%
  \else%
    \setlength{\unitlength}{\svgwidth}%
  \fi%
  \global\let\svgwidth\undefined%
  \global\let\svgscale\undefined%
  \makeatother%
  \begin{picture}(1,0.59215164)%
    \lineheight{1}%
    \setlength\tabcolsep{0pt}%
    \put(0,0){\includegraphics[width=\unitlength,page=1]{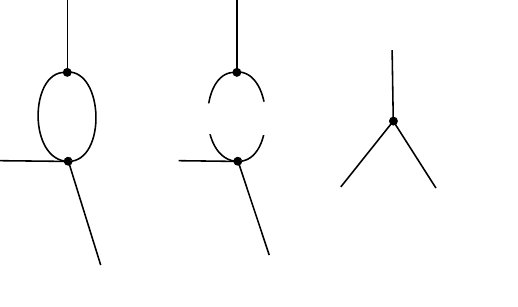}}%
    \put(0.4266936,0.0047368){\color[rgb]{0,0,0}\makebox(0,0)[lt]{\lineheight{1.25}\smash{\begin{tabular}[t]{l}$\nu(\Gamma)$\end{tabular}}}}%
    \put(0.72150457,0.00871183){\color[rgb]{0,0,0}\makebox(0,0)[lt]{\lineheight{1.25}\smash{\begin{tabular}[t]{l}$\pi_0(\Gamma)$\end{tabular}}}}%
    \put(0.12826911,0.00544923){\color[rgb]{0,0,0}\makebox(0,0)[lt]{\lineheight{1.25}\smash{\begin{tabular}[t]{l}$\Gamma$\end{tabular}}}}%
  \end{picture}%
\endgroup%
}
	\caption{On the definition of $\nu(\Gamma)$ and $\pi_0(\Gamma)$.}
	\label{nupicap}
\end{figure}
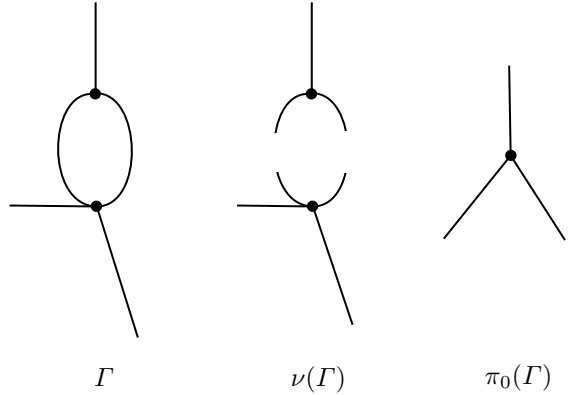

Under disjoint union, $\Graphs$ forms a symmetric monoidal category. 
One can now define a modular operad
 in any (higher) symmetric monoidal target category $\cat{S}$
as a symmetric monoidal functor (up to coherent homotopy) $\Graphs \to\cat{S}$.
While this offers a conceptually clear picture, using this definition at a technical level in the case that $\cat{S}$ is an arbitrary symmetric monoidal $\infty$-category can be intricate.
Luckily, the topological operads that we will be interested in are naturally obtained by taking the geometric realization of the nerve of a category-valued operad.
Therefore, we can concentrate on the case $\cat{S}=\Cat$ in which $\cat{S}$ is the symmetric monoidal bicategory of categories, functors and natural transformations. 
Let us record the definition in that case:
A \emph{modular operad} with values in the symmetric monoidal bicategory $\Cat$ of categories is
 a symmetric monoidal functor $\cat{O}:\Graphs \to \Cat$. Here $\Graphs$ is seen as symmetric monoidal bicategory with no non-identity 2-morphisms. As just mentioned, it is understood that the functor $\cat{O}$ does not need to have a strict functoriality and monoidality; instead, it is relaxed up to coherent homotopy, see \cite[Chapter~2]{schommerpries} for an overview of the coherence conditions for symmetric monoidal functors between symmetric monoidal bicategories. 
The symmetric monoidal category $\Graphs$ has a symmetric monoidal subcategory $\Forests$ with the same objects and forests as morphisms, i.e.\ disjoint unions of contractible graphs. A $\Cat$-valued \emph{cyclic operad}
is a symmetric monoidal functor $\cat{O}:\Forests \to \Cat$. Here we slightly deviate from the usual definition 
of the category $\Forests$ which would exclude corollas with no legs. This has the consequence that cyclic operads 
have `operations of arity $-1$ (total arity 0)'. This turns out to be useful later one and does not change the theory in an essential 
way. 

We should emphasize that in this description of \label{commentoperadicidentities}
cyclic and modular operads via graph categories, 
operadic identities (an operation $1_\cat{O}$ in total arity two 
that is neutral with respect to composition)
are a priori not included. 
These have to be included by hand \cite[Definition~2.3]{cyclic}, 
and for operads with values in a bicategory, such as $\Cat$, 
this involves several coherence conditions. 
In particular, an operadic identity is structure, 
not a property.
In this article, all (cyclic or modular) 
operads have by default an operadic identity.

The relevant operads for us will be the framed $E_2$-operad $\framed$
and the handlebody operad  $\Hbdy$
whose definition  we give now:
\begin{itemize}
	\item 
We define the \emph{(groupoid-valued) modular handlebody operad} as the symmetric monoidal functor 
$\Hbdy : \Graphs \to \Grpd$ sending a corolla $T$ with set of legs $\Legs(T)$
 to the groupoid $\Hbdy(T)$ whose objects are compact connected oriented three-dimensional
 handlebodies $H$ (we call them just  handlebodies for brevity) with $|\Legs(T)|$ many disks embedded in $\partial H$
  plus a parametrization of the disks, i.e.\ an orientation-preserving embedding $\psi : \sqcup_{\Legs(T)} \mathbb{D}^2 \to H$ (by this we mean an orientation-preserving diffeomorphism to the embedded disks). 
We define the morphisms in $\Hbdy(T)$ to be isotopy classes of orientation-preserving diffeomorphisms compatible with the  parametrizations. Therefore, the automorphism groups are 
the handlebody groups. 
The operadic composition in $\Hbdy$ is defined by the gluing of handlebodies along their boundaries.
More precisely, on the object level, the operadic composition glues handlebodies along a specified subset of the embedded disks; the parametrization of these disks is used to identify pairs of disks that we glue along. The definition on the object level is made such that mapping classes respecting the parametrizations can be glued as well, thereby affording the gluing operation on the morphism level.
A definition of handlebodies as a modular operad is also given in~\cite[Section~4.3]{giansiracusa}.
This definition essentially agrees with ours, but we should highlight that the definition in~\cite{giansiracusa} is essentially 1-categorical while we consider operads with values in bicategories, meaning that all sorts of coherence isomorphisms enter the picture. In addition, our conventions for operadic identities differ.

\item 
The \emph{(topological) framed $E_2$-operad} can be described by the spaces of embeddings of several disks into a bigger disks that are composed of scalings, translations and rotations; a detailed treatment can be found in \cite[Section~1.5]{wahlthesis} and \cite{salvatorewahl}. 
For us, however, it is more convenient to see $\framed$ as the categorical operad given by the genus zero part $\Hbdy_0$ of the handlebody operad. This also has the advantage that the cyclic structure is easier to understand. 
(The fact that this handlebody description of $\framed$ agrees with the traditional description of $\framed$ relies on the fact that $\framed$ can be seen as
cyclic operad of oriented genus zero surfaces and that the mapping class group of a genus zero handlebody agrees with the mapping class group of its boundary, see e.g.~\cite[Proposition~2.1]{hahe}.) 

In our conventions, $\framed$ has a contractible groupoid of arity $-1$ operations, namely the fundamental groupoid of the moduli space of closed surfaces of genus zero (this is, of course, just the sphere and its trivial mapping class group).
\end{itemize}
Next we briefly recall the definition of modular algebras over a modular operad $\cat{O}$ in $\Cat$:
For a symmetric monoidal bicategory $\cat{M}$ with unit $I$ and $X\in\cat{M}$, a non-degenerate symmetric pairing $\kappa :X\otimes X\to I$ is a 1-morphism that exhibits $X$ as its own dual object in $\cat{M}$ (i.e.\ there is a coevaluation $\Delta:I\to X\otimes X$ such that $\kappa$ and $\Delta$ fulfill the usual zigzag relations up to isomorphism) and that is a homotopy fixed point with respect to the $\mathbb{Z}_2$-action on $\cat{M}(X^{\otimes 2},I)$ coming from the symmetric braiding of $\cat{M}$. Now one can define the \emph{modular endomorphism operad of $(X,\kappa)$}. This is a modular operad $\End_\kappa^X : \Graphs \to \Cat$ sending a corolla $T$ with $n$ legs to $\cat{M}(X^{\otimes n},I)$. An $\cat{M}$-valued modular $\cat{O}$-algebra is an object $X\in\cat{M}$ with non-degenerate symmetric pairing $\kappa$ plus a symmetric monoidal transformation $A:\cat{O}\to \End_\kappa^X$ between symmetric monoidal functors between symmetric monoidal bicategories. 
In particular, $A$ consists of functors $A_T:\cat{O}(T)\to\cat{M}(X^{\otimes n},I)$, where $T$ is a corolla with $n$ legs,
and a natural isomorphism
\begin{equation}
\begin{tikzcd}
\mathcal{O}(T) \ar[dd,swap,"\mathcal{O}(\Gamma)"]  \ar[]{rr}{A_T } & &   \End_\kappa^X(T)  \ar{dd}{     \End_\kappa^X(\Gamma)  }  \ar[lldd, Rightarrow,"A_\Gamma"]  \\
& & \\
\mathcal{O}(T') \ar[swap]{rr}{A_{T'}} & &  \End_\kappa^X(T')
\end{tikzcd} 
\end{equation}
for any morphism $\Gamma:T\to T'$ in $\Graphs$. Since our operads have by default
operadic identities as explained on page~\pageref{commentoperadicidentities}, we additionally require all operadic algebras to be compatible with these operadic identities in the sense that
the map $A_{T_1}$ sends the operadic identity to the pairing
$\kappa \in \mathcal{M}(X^{\otimes 2}, I)$, which
serves as the operadic identity in the modular endomorphism operad. 
For operads with values in bicategories, it is implicit that this compatibility of
operadic algebras with the operadic identities is relaxed up to coherent isomorphism; it becomes structure. The details are presented in \cite[Definition~2.13]{cyclic}.
Cyclic algebras are defined analogously with $\Graphs$ replaced by $\Forests$, see \cite[Section~2.4]{cyclic}.

Bicategorical algebras over a cyclic or modular operad form a 2-groupoid~\cite[Proposition 2.18]{cyclic}. We briefly sketch the definition of 1-morphisms and 2-morphisms\label{refdefbicatalg2grpd}
 for algebras over a modular operad. The cyclic case is completely analogous;
 for more details, we refer to~\cite[Section~2.4]{cyclic}. A \emph{1-morphism} $F\colon A \to A' $ between algebras $A$ and $A'$
  over a modular operad $\cat{O}$ consists of a 1-morphism $F\colon X\to X'$ in $\cat{M}$ between the underlying objects of $A$ and $A'$, respectively, together with 
\begin{itemize}
	\item a 2-isomorphism 
	\begin{equation}\label{eqncoh1}
	\begin{tikzcd}
		X\otimes X \ar[rd, "\kappa", swap] \ar[rr, "F\otimes F"] & \ar[d,Rightarrow, "F_\kappa"] & X' \otimes X' \ ,  \ar[ld, "\kappa'"] \\ 
		 & I & 
	\end{tikzcd}
	\end{equation} 
\item and for all operations $o\in \cat{O}(T)$ a 2-isomorphism 
 \begin{equation}\label{eqncoh2}
\begin{tikzcd}
	X^{\otimes n} \ar[rd, "A_{o}", swap] \ar[rr, "F^{\otimes n}"] & \ar[d,Rightarrow, "F_o"] & {X'}^{\otimes n} \ .   \ar[ld, "A'_{o}"] \\ 
	& I & 
\end{tikzcd}
 \end{equation}
\end{itemize} 
These 2-isomorphisms
 are natural with respect to the morphisms in $\cat{O}(T)$, 
and compatible with the symmetric structure of the pairings and the operadic composition. For algebras over a operad with operadic identity,
 we require the 2-morphism $F_\kappa$ from~\eqref{eqncoh1}
  to be induced by $F_{1_\cat{O}}$
  from~\eqref{eqncoh2}, where $1_\cat{O}$ is the operadic identity.

A \emph{2-morphism} $\varphi\colon F\to F'$ between 1-morphisms $F,F'\colon A \to A' $ of algebras over a modular operad $\cat{O}$ 
(the 2-isomorphisms~\eqref{eqncoh1} and~\eqref{eqncoh2} are part of the data, but are suppressed from the notation)
is a 2-morphism $\varphi\colon F\longrightarrow F'$ in $\cat{M}$ such that 
 \begin{equation}
	\begin{tikzcd}
		& \ar[d, Rightarrow, "\varphi^{\otimes n}" near start, shorten >=2ex] & \\
		X^{\otimes n} \ar[rd, "A_{o}", swap] \ar[rr, "F^{'\otimes n}"] \ar[rr, "F^{\otimes n}", bend left=60] & \ \ar[d,Rightarrow, , "F_o'"] & {X'}^{\otimes n}  \ar[ld, "A'_{o}"] \\ 
		& I & 
	\end{tikzcd} 
= 
	\begin{tikzcd}
		& & \\
		X^{\otimes n} \ar[rd, "A_{o}", swap] \ar[rr, "F^{\otimes n}"] & \ar[d,Rightarrow, "F_o"] & {X'}^{\otimes n}  \ar[ld, "A'_{o}"] \\ 
		& I & 
	\end{tikzcd}
\end{equation}  
for every operation $o$ in $\cat{O}$.
  
We are now ready to state the definition of an ansular functor:

\begin{definition}\label{defansularfunctordef}
	An \emph{ansular functor}
	with values in a symmetric monoidal bicategory $\cat{M}$
	is an $\cat{M}$-valued modular algebra over the handlebody operad. 
\end{definition}

In this compact language, we may also state the main goal of this article: \emph{the classification of all ansular functors}.

\begin{remark}[Unpacking the definition]\label{rem: concrete description ansular functor}
	An ansular functor $A$ in $\Lexf$ consists of a 
	finite category $\cat{C}$ with non-degenerate pairing $\kappa:\cat{C}\boxtimes \cat{C}\to \fVect$ and for every handlebody $H$ with $n+1$
	disks embedded into its boundary a linear functor $A(H)\colon \cat{C}^{\boxtimes n+1}\to \fVect$ on which the handlebody group of $H$ acts via natural isomorphisms. For every collection of objects $X_0,\dots , X_n\in \cat{C}$,
	the evaluation of $A$ on $X_0\boxtimes \dots \boxtimes X_n$ gives a vector space 
	$A(H;X_0,\dots , X_n)\in \fVect$ which is equipped with a linear representation of the handlebody group of $H$. 
	These representations glue together via the coevaluation $\Delta : \fVect \to \cat{C}\boxtimes\cat{C}$
	for $\cat{C}$; this corresponds to a coend over the labels at the gluing boundary
	(for the details, see the excision result in \cite[Theorem~6.4]{cyclic}).
	 An ansular functor can be understood as the collection of all these handlebody  representations compatible with cutting and gluing. This is very similar in spirit to the notion of a modular functor; we expand on this in Remark~\ref{remmodular}.
	Using the non-degenerate pairing, the functors $A(H)$ gives rise to linear functors 
	\begin{align}
		\cat{C}^{\boxtimes i} \to \cat{C}^{\boxtimes j} \label{eqnAinputoutput}
	\end{align}    
	for any decomposition of the $n+1$ boundary disks into $i$ `incoming' and $j$ `outgoing' disks.
	The corresponding operation
	in $\Hbdy$ does not come with such a distinction, but if needed, one may always deliberately choose one.
	 We suggestively write $A(H\colon \sqcup_{i}\mathbb{D}^2\to \sqcup_{j}\mathbb{D}^2)$ for the functors~\eqref{eqnAinputoutput}.
\end{remark}

\begin{remark}[Relation to modular functors]\label{remmodular}
	As mentioned in the introduction, a modular functor \cite{Segal,ms89,turaev,tillmann,baki} is a consistent system of mapping class group representations.
	It can be formalized as follows, see \cite[Setion~7]{cyclic}: Consider the modular surface operad $\Surf$ that one obtains by replacing handlebodies with surfaces in the definition of $\Hbdy$; this modular operad was one of the motivations for the invention of modular operads in~\cite{gkmod}. 
	A modular functor can then be defined as a modular algebra over certain (central) extensions $\Surf^\catf{c} \to \Surf$ of the modular surface operad $\Surf$. The class of extensions that one allows is a choice, and different authors might make slightly different ones.
	 In any case, the extensions will be made such that the boundary map
	 $\partial : \Hbdy \to \Surf$, the map of modular operads that takes the boundary of a handlebody and converts the embedded disks in its boundary into boundary components of the boundary surface, lifts in a canonical way to a map $\Hbdy \to \Surf^\catf{c}$ of modular operads. Therefore, any modular functor gives, by restriction along the (lift of the) boundary map, rise to an ansular functor. 
	\end{remark}

\spaceplease
\section{The modular envelope and dihedral homology\label{modenvsec}}
In this section, we discuss one of the main tools for the classification of ansular functors, namely the modular envelope and its connection to dihedral homology. The results of this section are somewhat technical and require a little bit of basic homotopy theory. They will be needed in some of the proofs in the subsequent sections, but are dispensable for the understanding of the statement of the results in these sections. The reader who wants to take a shortcut may just skim through the notion of the modular envelope (until equation~\eqref{eqnenvenvint}), take note of the standard facts on diffeomorphism and mapping class groups in Remark~\ref{remdiffgroup} and move on to Section~\ref{secmodext}. One may then later come back to the results in this section once they are needed.

A cyclic operad can be `freely' completed to a modular operad, its so-called \emph{modular envelope} \cite{costello}.
More precisely,
the modular envelope $\Envint \cat{O}:\Graphs \to \Cat$ of a $\Cat$-valued cyclic operad $\cat{O}:\Forests\to\Cat$ is defined as the left Kan extension of $\cat{O}$ along the inclusion $\ell : \Forests \to \Graphs$.
The subscript `$\int$' is supposed to remind us 
 that this left Kan extension is performed in the homotopically correct way.
In the category-valued case, this means concretely that $(\Envint \cat{O})(T)$ for $T\in\Graphs$ is defined as the \emph{Grothendieck construction} 
\begin{align}
	(\Envint \cat{O})(T) = \int \left(\ell/T \to \Forests \ra{\cat{O}}\Cat \right) \ , \label{eqngc}
	\end{align}
i.e.\ as the category of pairs formed by an element $T' \ra{\Gamma} T$ in the slice category of $\ell$ over $T$ and some $o\in \cat{O}(T')$
(recall e.g.~from
\cite[Section~I.5]{maclanemoerdijk} that the Grothendieck construction of a functor $F:\cat{C}\to\Cat$ is the category $\int F$ that has as objects pairs $(c,x)$ formed by all objects $c\in \cat{C}$ together with an object $x\in F(c)$).
We will not call this construction \emph{derived} modular envelope because the word `derived' might create confusion in the category-valued case.
In fact, the Grothendieck construction can be understood as an oplax colimit.
By Thomason's Theorem \cite[Theorem~1.2]{thomason} it translates to a homotopy colimit in topological spaces after taking nerve (denoted by $B$) and geometric realization (denoted by $|-|$).
As a consequence, $|B\Envint \cat{O}|$ is really the `derived modular envelope'
$\mathbb{L}\Env|B\cat{O}|$ of the topological cyclic operad $|B\cat{O}|$ as it appears in \cite{costello,giansiracusa};
\begin{align}
	|B\Envint \cat{O}|\simeq
	\mathbb{L}\Env|B\cat{O}| \ . \label{eqnenvenvint}
	\end{align}
For a corolla $T$,
the category $\ell / T $ has the following concrete description: Its objects 
are connected graphs with an identification of their legs with $T$. 
Morphisms are given by collapsing of subtrees and symmetries of the 
graph. We denote this category by $\catf{Gr}_{\operatorname{conn}}(T)$. 
We 
denote by $\catf{Gr}_{\operatorname{conn}}^{\operatorname{a}}(T)$ its full subcategory
spanned by graphs containing no 
vertex of valence one. There is one exception:
If $T$ is the corolla with one leg, then we allow $T$ 
itself as object in $\catf{Gr}_{\operatorname{conn}}^{\operatorname{a}}(T)$.

\begin{lemma}\label{lemmaiota}
For any corolla $T$, the inclusion $\iota:\catf{Gr}_{\operatorname{conn}}^{\operatorname{a}}(T) \to \catf{Gr}_{\operatorname{conn}}(T)$ is homotopy final. 
\end{lemma}

Recall that a functor $F:\cat{C}\to\cat{D}$ is called \emph{homotopy final} if the realization $|B(d/F)|$ of the nerve of the slice of $F$ under any $d\in \cat{D}$ is contractible. We are using here the terminology of \cite[Section~8.5]{riehl}. 
 
\begin{proof}
For a connected graph $\Gamma\in \catf{Gr}_{\operatorname{conn}}(T)$,
we need to show that the slice $\Gamma / \iota $ is contractible. To this end,
 let $\Gamma^{\operatorname{a}}$ be the graph 
constructed from $\Gamma$ by collapsing
 all vertices with one leg until we 
get an element of $\catf{Gr}_{\operatorname{conn}}^{\operatorname{a}}(T)$. The collapsing procedure describes a morphism 
$\Omega_\Gamma \colon \Gamma \to \Gamma^{\operatorname{a}} \in \catf{Gr}_{\operatorname{conn}}(T)$, see Figure~\ref{fig:1} for a sketch.
 We can now observe that $\Omega_\Gamma$ is an initial object in $\Gamma / \iota $.
 This implies
that $|B(\Gamma / \iota) |$
is contractible.    
\end{proof}

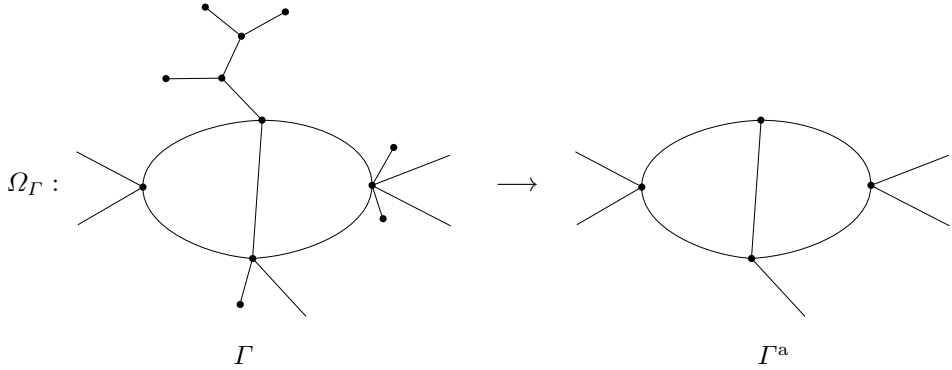
\begin{figure}[h]
	\begin{center}
		\begin{overpic}[
			scale=0.4]{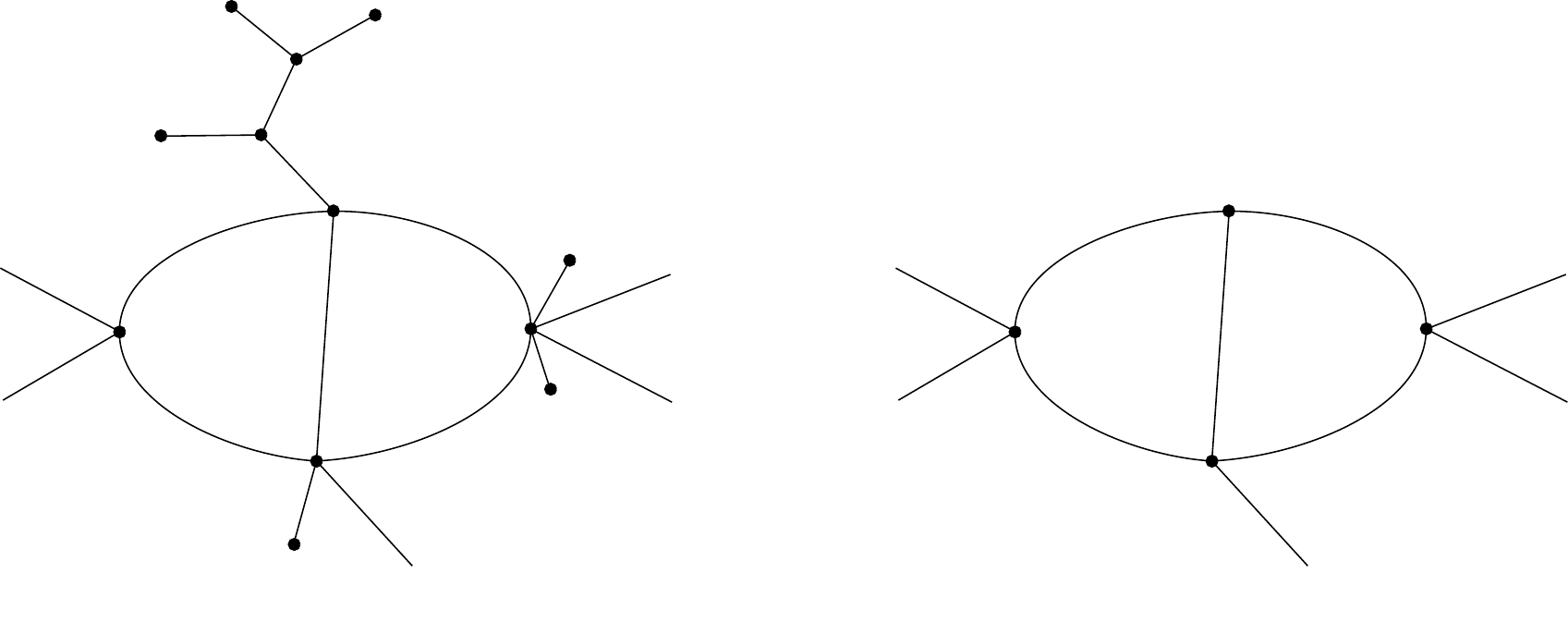}
			\put(-8,18){$\Omega_\Gamma \colon $}
			\put(48,18){$\longrightarrow$} 
			\put(18,-2){$\Gamma$}
			\put(78,-2){$\Gamma^\text{a}$}
		\end{overpic}
	\end{center}
	\caption{Sketch for the map $\Omega_\Gamma$.}
	\label{fig:1}
\end{figure}

This has the consequence that arity zero operations can essentially be
 ignored for the computation of the modular envelope:

\begin{proposition}\label{Prop:no0vertices}
Let $\mathcal{O}$ be a category-valued cyclic operad. For all corollas $T$,
 the inclusion
 $\iota:\catf{Gr}_{\operatorname{conn}}^{\operatorname{a}}(T) \to \catf{Gr}_{\operatorname{conn}}(T)$
 induces
  a homotopy equivalence
\begin{align}
\left|B \int \left( \catf{Gr}_{\operatorname{conn}}^{\operatorname{a}}(T) \to \Forests \ra{\cat{O}} \Cat \right) \right| \ra{\simeq}
|B \Envint \cat{O} | (T) \ .         \label{eqnmapgrothendieckconstructions}
 \end{align}
\end{proposition}

\begin{proof}
	We obtain the map~\eqref{eqnmapgrothendieckconstructions} by
	applying $|B-|$
	(nerve and geometric realization)
	 to the functor \begin{align} \int \left( \catf{Gr}_{\operatorname{conn}}^{\operatorname{a}}(T) \to \Forests \ra{\cat{O}} \Cat \right) \to \Envint \cat{O}(T)\end{align} induced by  $\iota:\catf{Gr}_{\operatorname{conn}}^{\operatorname{a}}(T) \to \catf{Gr}_{\operatorname{conn}}(T)$.
	Thanks to Thomason's Theorem~\cite[Theorem~1.2]{thomason}, this can equivalently be described as the map
	\begin{align} &\hocolim \left(\catf{Gr}_{\operatorname{conn}}^{\operatorname{a}}(T) \to \Forests \ra{\cat{O}}\Cat \ra{| B -|}\Top \right) \\ \to\quad  
	&	\hocolim \left(\catf{Gr}_{\operatorname{conn}}(T) \to \Forests \ra{\cat{O}}\Cat \ra{|B -|}\Top \right) 
		\end{align}
	induced by $\iota:\catf{Gr}_{\operatorname{conn}}^{\operatorname{a}}(T) \to \catf{Gr}_{\operatorname{conn}}(T)$.
	This map is an equivalence 
	since $\iota$ is homotopy final
	by Lemma~\ref{lemmaiota}.
	\end{proof}

		For any corolla $T$ and a non-negative integer $g$, denote by $\catf{Gr}_{\operatorname{conn}}^{\operatorname{a},g}(T) \subset \catf{Gr}_{\operatorname{conn}}^{\operatorname{a}}(T)$ the full subcategory spanned by graphs whose first Betti number is $g$.
	Then $\catf{Gr}_{\operatorname{conn}}^{\operatorname{a},g}(T)$ is connected and
	\begin{align}
	\catf{Gr}_{\operatorname{conn}}^{\operatorname{a}}(T) = \bigsqcup_{g\ge 0} \catf{Gr}_{\operatorname{conn}}^{\operatorname{a},g}(T) \ . 
	\end{align}
	As an immediate consequence, we find
	\begin{align}
	|B \Envint \cat{O}| (T) &\simeq \bigsqcup_{g \ge 0} |B \Envintg \cat{O}|(T)  \label{eqndecompgenus} \\
	 \text{with}\quad \Envintg \cat{O}(T) &:= \int \left(  \catf{Gr}_{\operatorname{conn}}^{\operatorname{a},g}(T) \subset  \catf{Gr}_{\operatorname{conn}}(T) \to \Forests \ra{\cat{O}} \Cat  \right) \ .   
	\end{align}
	We call $\Envintg \cat{O}(T)$
	the
	\emph{genus $g$ contribution}
	to 	$\Envint \cat{O}(T)$. In the sequel, it will be necessary to understand the genus one contribution in detail. For this reason,
	 the category $\catf{Gr}_{\operatorname{conn}}^{\operatorname{a},1}(\bullet)$ will be of particular importance.
	 Here we use the additional notation $\bullet$ for the corolla with no legs. 
	 
	The category $\catf{Gr}_{\operatorname{conn}}^{\operatorname{a},1}(\bullet)$ can be described using 
	\emph{Connes' cyclic category $\Lambda$} \cite{connes}
	whose definition we briefly recall:
	The objects of $\Lambda$ are the natural numbers $\mathbb{N}_0$. The object corresponding to $n\ge 0$ will be denoted by $[n]$.
	One defines a morphism $f:[n]\to [m]$ as an equivalence class of set-theoretic maps $f: \mathbb{Z}\to\mathbb{Z}$ subject to the condition $f(i+n+1)=f(i)+m+1$. The equivalence relation identifies $f$ and $g$ if $f-g$ is a constant multiple of $m+1$. 
	The simplex category $\Delta$ sits inside $\Lambda$ as a subcategory.
	The usual description of $\Delta$ in terms of face maps $\delta_i\colon [n-1] \to [n]$ for $ 0\leq i \leq n $ and degeneracy maps $\sigma_j\colon
	[n+1] \to [n] $ for $ 0\leq j \leq n$ can be extended to a description of $\Lambda$ by including cyclic permutations $\tau_n \colon 
	[n] \to [n]$ with $\tau_n^{n+1}=\id_{[n]}$. In addition to the usual simplicial relations
	between the face and degeneracy maps, there are compatibility relations with the cyclic permutations:
\begin{align}
\tau_n \delta_i &= \delta_{i-1} \tau_{n-1}\ ,  \quad  1\leq i \leq n \\
\tau_n \delta_0 &= \delta_n\ ,  \\ 
\tau_n \sigma_i &= \sigma_{i-1} \tau_{n-1} \ , \quad  1\leq i \leq n \\ 
\tau_n \sigma_0 &= \sigma_n \tau_{n+1}^2 \  .
\end{align}
By omitting the degeneracy maps we obtain a subcategory 
 $\vec\Lambda \subset \Lambda$. The following statement is most likely well-known. We give a proof in lack of a reference. 

\begin{lemma}\label{lemmaLambdavec}	The inclusion $\vec\Lambda \subset \Lambda$ is homotopy initial, thereby making the inclusion $\vec\Lambda^\opp \subset \Lambda^\opp$ 	homotopy final.	
\end{lemma}	
\begin{proof}	For the simplex category $\Delta$, consider the inclusion $\vec\Delta \subset \Delta$ of the subcategory without the degeneracy maps. Since $\vec\Delta \subset \Delta$ is well-known to be homotopy initial, see e.g.\ \cite[Example~8.5.12]{riehl},	it suffices to prove that the canonical inclusion	
\begin{align}	
I_{[n]}:	\vec\Delta / [n] \to \vec\Lambda / [n]	
\end{align}	
induces a homotopy equivalence	
\begin{align}|B(\vec\Delta / [n])| \ra{\simeq} |B(\vec\Lambda / [n])| \ .\label{eqnhomotopyequivalence}	
\end{align}	In fact, any map $f : [m]\to [n]$ in $\Lambda$ uniquely factors as \begin{align}\label{eqnfactorization} 
[m] \ra{z(f)} [m] \ra{f_\Delta} [n] \ , 
\end{align} 
where $z(f)$ is an automorphism in $\Lambda$	(and therefore in $\vec\Lambda$) and $f_\Delta : [m]\to [n]$ lies in $\Delta$ (we see here $\Delta$ as subcategory of $\Lambda$).	 We now consider the slice category $f / I_{[n]}$: Objects in this category are factorizations $[m] \ra{h} [\ell] \ra{g} [n]$ of $f$, where $g$ is in $\Delta$ and $h$ in $\vec\Lambda$. The factorization~\eqref{eqnfactorization} is initial in $f / I_{[n]}$ proving that $|B(f / I_{[n]})|$ is contractible. Now $I_{[n]}$ is homotopy final and therefore \eqref{eqnhomotopyequivalence} a homotopy equivalence by Quillen's Theorem~A. \end{proof}

The cyclic category can be extended to the \emph{dihedral category} \cite{loday,spalinski}.
We use the description from \cite[Section~2]{mwdiff}:
The group $\mathbb{Z}_2$ acts on $\Lambda$ by an involution 
$r: \Lambda \to \Lambda$ that we call the \emph{reversal functor}.
It is the identity on objects and sends any morphism $f:[n]\to[m]$  in $\Lambda$ represented by a map $f:\mathbb{Z}\to\mathbb{Z}$ to the morphism $r(f)$ represented by $p \mapsto m-f(n-p)$.
The $\mathbb{Z}_2$-action gives us a functor $*\DS \Z_2 \to \Cat$ sending $*$ to $\Lambda$ and the generator of $\mathbb{Z}_2$ to the reversal functor.
We can obtain the dihedral category as the Grothendieck construction of this functor. We denote it by $\Lambda \rtimes \mathbb{Z}_2$.
Similarly, one can define $\vec\Lambda \rtimes \mathbb{Z}_2$ and $\vec\Lambda^\opp \rtimes\mathbb{Z}_2= (\vec\Lambda \rtimes \mathbb{Z}_2)^\opp$.

We may now make the following elementary observation:

	\begin{lemma}\label{Lemconnectedgraphs}
	There is a canonical equivalence of categories
	$\catf{Gr}_{\operatorname{conn}}^{\operatorname{a},1}(\bullet) \simeq \vec\Lambda^\opp \rtimes \mathbb{Z}_2$. 
\end{lemma} 

\begin{proof}
	Consider the functor $ \Psi \colon \vec\Lambda^\opp \rtimes 
	\mathbb{Z}_2 \to 
	\catf{Gr}_{\operatorname{conn}}^{\operatorname{a},1}(\bullet) $ 
	sending $[n]$ to the graph $\Psi([n])$ with $n+1$ bivalent vertices
	labeled by $0,\dots , n$ arranged on a circle, the generator $\tau_n$ to the cyclic permutation of 
	the vertices of $\Psi([n])$, $\delta_i$ for $0\le i\le n$ to the morphism which  
	collapses 
	the edge between vertex $i$ and $i+1 \mod (n+1)$, and $-1\in \Z_2$ to 
	the automorphism of $\Psi([n])$ sending vertex $i$ to the vertex $n+1-i \mod 
	(n+1) $.

	The functor $\Psi$ is clearly essentially surjective.
	To see that it is fully faithful, note that every morphism $\Psi([n])\to \Psi([m])$ in 
	$\catf{Gr}_{\operatorname{conn}}^{\operatorname{a},1}(\bullet)$ can be
	written uniquely as the composition of a subtree collapse (which are always given by compositions of images of $\delta_i$'s) with an automorphism 
	of $\Psi([m])$. The same type of factorization holds for the morphisms of $\vec\Lambda^\opp \rtimes \mathbb{Z}_2$ (for $\vec \Lambda$, this is a standard fact, and it clearly carries over to the semidihedral category 
	$\vec\Lambda^\opp \rtimes \mathbb{Z}_2$).
	It follows from comparing these factorizations and fact that both categories have the same automorphism groups, namely $\Z_{n+1}\rtimes \Z_2$ for $n\ge 0$, that $\Phi$ is fully faithful.    
\end{proof}

The dihedral category allows us to define Loday's dihedral homology. Let us briefly recall the definition:
Let $A$ be an associative algebra in a symmetric monoidal $\infty$-category $\cat{S}$ (for us, spaces or chain complexes are sufficient). 
By means of the multiplication and the unit of $A$, we can build the following simplicial object in $\cat{S}$:
\begin{equation}
\begin{tikzcd}
\dots \ar[r, shift left=6]  \ar[r, shift left=2]
\ar[r, shift right=6]  \ar[r, shift right=2]
& 	\displaystyle A^{\otimes 3}
\ar[l, shift left=4]  \ar[l]
\ar[l, shift right=4]  
\ar[r, shift left=4] \ar[r, shift right=4] \ar[r] & \displaystyle A^{\otimes 2} \ar[r, shift left=2] \ar[r, shift right=2]
\ar[l, shift left=2] \ar[l, shift right=2]
& \displaystyle A\ ,  \ar[l] \\
\end{tikzcd}\label{eqnhochschildgenus1}
\end{equation}
This is sometimes called the Hochschild object. If $\cat{S}$ is cocomplete, its colimit over $\Delta^\opp$ is defined as the \emph{Hochschild homology of $A$}
(it is standard to use the word `homology' here even if we are not necessarily working in chain complexes).  
As observed by Connes and Tsygan \cite{connes,tsygan}, 
\eqref{eqnhochschildgenus1}~extends to a functor $\Lambda^\opp \to \cat{S}$, i.e.\ to a \emph{cyclic object}.
Taking the homotopy colimit over $\Lambda^\opp$ gives us the \emph{cyclic homology} of $A$. The full $\infty$-construction was given by Nikolaus-Scholze \cite[Appendix~B]{ns}.

Now assume that $A$ comes with an involution through anti algebra maps.
In this case, \eqref{eqnhochschildgenus1} extends to a functor \begin{align}\label{eqndihedralobject0}
	\Lambda^\opp \rtimes \mathbb{Z}_2 \to \cat{S} \ , 
	\end{align} i.e.\ a \emph{dihedral object}, and its homotopy colimit is defined as the \emph{dihedral homology of $A$}; it was introduced by Loday in~\cite{loday}. We use the notation $\DH(A)$ for the dihedral homology of $A$. 
Since Loday did not construct this object in the setting of $\infty$-categories along the lines of
\cite[Appendix~B]{ns}, let us briefly comment on how the $\infty$-categorical dihedral object~\eqref{eqndihedralobject0} can be constructed:
 Denote by $\operatorname{Env}(\As)$ the symmetric monoidal category with finite sets as objects and morphisms being  maps together with a linear ordering on all preimages. This is the envelope of the associative operad, and hence symmetric monoidal functors out of $\operatorname{Env}(\As)$ are exactly associative algebras, see e.g.~\cite[Section~2.1]{afh}. There is a $\Z_2$-action on $\operatorname{Env}(\As)$ reversing the ordering on all the fibers. An algebra with an anti involution in an $\infty$-category $\cat{S}$ can now be defined as a $\Z_2$-equivariant symmetric monoidal functor $A\colon \operatorname{Env}(\As) \to \cat{S}$. The usual map $\Lambda^{\opp}\to \operatorname{Env}(\As)$ is $\Z_2$-equivariant. The Grothendieck construction $\Lambda^\opp \rtimes \mathbb{Z}_2$ is a oplax colimit and hence the $\Z_2$-equivariant functor $\Lambda^\opp\to \cat{S}$ induces the dihedral object $\Lambda^\opp \rtimes \mathbb{Z}_2\to \cat{S}$.

	\begin{definition}
		For every groupoid-valued cyclic operad $\cat{O}$,
		we denote by $\algo$ the \emph{algebra $\cat{O}(T_1)$ of unary operations}. The multiplication is the operadic composition.
		We agree to see $\algo$ as algebra in spaces.
		\end{definition}
	
	\begin{remark}\label{remdihhom}
		The cyclic symmetry of the cyclic operad $\cat{O}$ endows $\algo$ with a $\mathbb{Z}_2$-action by anti algebra maps. For this reason, its dihedral homology $\DH(\algo)$ can be considered.
		Since the inclusion $\vec \Lambda^\opp \rtimes \mathbb{Z}_2\subset \Lambda^\opp \rtimes \mathbb{Z}_2$ is homotopy final (this follows from the fact that $\vec\Lambda^\opp \subset \Lambda$ is homotopy final, see Lemma~\ref{lemmaLambdavec}), we may compute $\DH(\algo)$ as 
		\begin{align}
		\DH(\algo) = \hocolimsub{[n] \in \vec\Lambda^\opp\rtimes \Z_2}\algo^{\times (n+1)} \ . 
		\end{align}
		\end{remark}

	\begin{theorem}\label{thmgenusonegeneral}
		Let $\mathcal{O}\colon \Forests \to \Cat$ be a cyclic operad.
		Then 
		there is a canonical homotopy equivalence
		\begin{align} 	|B\Envintone 
		\cat{O}(\bullet)| \simeq \DH(\algo)=\hocolimsub{[n] \in \vec\Lambda^\opp\rtimes \Z_2} \algo^{\times (n+1)} \ , \end{align} i.e.\ the genus one contribution
		$|B\Envintone 
		\cat{O}(\bullet)|$
		to the arity $-1$ operations of the modular envelope
		 is homotopy equivalent to the dihedral homology of the algebra of unary operations. 
	\end{theorem}

	\begin{proof}
		Restriction of Proposition~\ref{Prop:no0vertices} to the `genus one part'
		(i.e.\ the component for $g=1$ in~\eqref{eqndecompgenus})
		 gives us a homotopy equivalence \begin{align} |B\Envintone 
		\cat{O}(\bullet)| \simeq \hocolim \left(   \catf{Gr}_{\operatorname{conn}}^{\operatorname{a},1}(\bullet)  \subset  \catf{Gr}_{\operatorname{conn}}^{\operatorname{a}}(\bullet)  \to \Forests \ra{|B\cat{O}|} \Top\right) \ . 
		\end{align}
		If we precompose the diagram $\catf{Gr}_{\operatorname{conn}}^{\operatorname{a},1}(\bullet)\to \Top$ on the right hand side with the equivalence $\Psi :  \vec\Lambda^\opp \rtimes 
		\mathbb{Z}_2 \to 
		\catf{Gr}_{\operatorname{conn}}^{\operatorname{a},1}(\bullet)$ from Lemma~\ref{Lemconnectedgraphs}, the 
		homotopy colimit remains the same up to equivalence, and we obtain the diagram $\vec\Lambda^\opp\rtimes \Z_2\to \Top$ sending $[n]$ to $\algo^{n+1}$. 
		By definition the homotopy colimit of this 
		diagram is the dihedral homology of $\algo$ (Remark~\ref{remdihhom}).
		This implies the assertion.
		\end{proof}
	
In the case that the cyclic operad $\cat{O}$
is the framed $E_2$-operad $\framed$,
 we obtain the following:

\begin{theorem}\label{thmGenus1envelope}
	There is a homotopy equivalence of topological spaces
	\begin{align}
		|B\Envintone 
		\framed(\bullet)|     \simeq   B \Diff (H_{1,0}) \ \ ,
	\end{align}
	where $\Diff(H_{1,0})$ is the topological group of diffeomorphisms of the solid closed torus $H_{1,0}=\mathbb{S}^1\times\mathbb{D}^2$. 
\end{theorem}

\begin{proof}
By definition of $\framed$, we have $\mathfrak{A}_{\framed}=\mathbb{S}^1$ with trivial $\mathbb{Z}_2$-action on $\mathbb{S}^1$.
 Therefore, Theorem~\ref{thmgenusonegeneral} implies that
 $	|B\Envintone 
 \framed(\bullet)| $ is the dihedral homology of the topological algebra $\mathbb{S}^1$ with trivial $\mathbb{Z}_2$-action.
 This dihedral homology was proven to be
 homotopy equivalent to $B \Diff (H_{1,0})$ in \cite[Theorem~2.2]{mwdiff}.
\end{proof}

\begin{remark}\label{remdiffgroup}
It is well-known that $\Diff (H_{1,0})\simeq \mathbb{T}^2 \rtimes (\mathbb{Z}\times \mathbb{Z}_2)$, where $\Map(H_{1,0})\cong \mathbb{Z}\times \mathbb{Z}_2$ is the mapping class group of the solid closed torus. This is a consequence of results in \cite{gramain,hatcher,Wajnryb}; a 
 detailed recollection is given in \cite[Section~2]{mwdiff}.
A generator for the $\mathbb{Z}$-factor of $\Map(H_{1,0})$ is the 
Dehn twist $T$ along any properly embedded disk in $H_{1,0}$; a generator for the $\mathbb{Z}_2$-factor
is the rotation by $\pi$ around any axis in the plane in which $H_{1,0}$ lies.
For later use, let us just recall that we can see $\Map(H_{1,0})$ as a subgroup of the mapping class group $\Map(\mathbb{T}^2)\cong \text{SL}(2,\mathbb{Z})$ and hence as a $2\times 2$-matrix.
Under the inclusion $\Map(H_{1,0})\subset \Map(\mathbb{T}^2)\cong \text{SL}(2,\mathbb{Z})$, the 
generators $T$ and $R$ are mapped as follows:
\begin{align}
\label{eqnTandRmatrix}	T \mapsto \begin{pmatrix} 1 & 0 \\ 1 & 1 \end{pmatrix}\ , \quad R \mapsto \begin{pmatrix} -1 & \phantom{-}0\\\phantom{-}0&-1 \end{pmatrix} \ . \end{align}
\end{remark}

\section{The modular extension Theorem for cyclic algebras\label{secmodext}}
For the proof of the main result of this article, we need to
 understand modular algebras over the (derived) modular envelope of a cyclic operad $\cat{O}$.
Since the modular envelope is designed to be the left adjoint to the forgetful functor from modular to cyclic operads, one should expect that
the modular algebras over the modular envelope of $\cat{O}$ are just cyclic $\cat{O}$-algebras. The details behind this, however, are quite subtle because they depend on the category in which the modular envelope is computed: in topological spaces via homotopy colimits, or in categories via the Grothendieck construction. 
Indeed,
modular algebras over the  categorical
modular envelope 
 $\Envint \cat{O}$ (built using the Grothendieck construction)
  are \emph{not} equivalent to cyclic $\cat{O}$-algebras. 
Luckily, a slightly modified statement turns out to be true. Presenting this statement is the purpose of this section.

When passing from a category-valued operad to an operad valued in spaces,
 we apply the functor $|B-|$ to the categories of operations. To pass back to categories, we apply the functor $\Pi$ sending a space to its fundamental groupoid. For a general category $\cat{C}$, the groupoid $\Pi |B\cat{C}|$ is the localization of $\cat{C}$ at all its morphisms, i.e.\ it is the initial category with a map from $\cat{C}$ to it which inverts all morphisms. This localization is characterized by the universal property that specifying a functor out of it is equivalent to specifying a functor out of $\cat{C}$ which sends all morphisms of $\cat{C}$ to isomorphisms. In the following, we will use this universal property of $\Pi |B\cat{C}|$. There are smaller models for this localization (see e.g.~\cite[Chapter~2.1]{riehl}).

For a symmetric monoidal bicategory $\cat{M}$
and a category-valued cyclic operad $\cat{O}$,
let $A$ be a cyclic algebra, i.e.\ a symmetric monoidal 
transformation $A:\cat{O}\to\End_\kappa^X$.
Moreover, let $T\in\Graphs$ and $(\widetilde T,\Gamma)\in \ell /T$, where $\ell:\Forests\to\Graphs$
is the inclusion.
By \cite[Proposition~7.1]{cyclic} the functor
\begin{align}
\cat{O}(\widetilde T) \ra{A_{\widetilde T}}\End_\kappa^X(\widetilde T) \ra{\End_\kappa^X(\Gamma)} \End_\kappa^X(T)
\end{align}
induces a functor $(\Envint \cat{O})(T)\to \End_\kappa^X(T)$. This yields on $(X,\kappa)$ the structure of a modular $\Envint \cat{O}$-algebra. It is called the \emph{modular extension} of $A$ and
denoted by $\widehat{A}$. If $\cat{O}$ is groupoid-valued, then the functors
$(\Envint \cat{O})(T)\to \End_\kappa^X(T)$ that $\widehat{A}$ consists of send all morphisms to isomorphisms    (although $\Envint\cat{O}$ is generally \emph{not} groupoid-valued).
Therefore, $\widehat{A}$ can be seen as a modular algebra over the modular operad $\Pi |B \Envint\cat{O}|$ (the localization of $\Envint \cat{O}$ at all morphisms as explained above). 
It is a straightforward consequence of the definition of 1-morphisms and 2-morphisms of cyclic and modular algebras (page~\pageref{refdefbicatalg2grpd})
that 1-morphisms and 2-morphisms between cyclic $\cat{O}$-algebras extend to 1-morphisms and 2-morphisms between the respective modular extensions.
In other words,
the modular extension
yields a functor of 2-groupoids
\begin{align}
\widehat{-} : \CycAlg (\cat{O})\to\ModAlg(\Pi |B \Envint\cat{O}|) \ . 
\end{align}
The fact that cyclic and modular algebras with values in a symmetric monoidal bicategory form 2-groupoids is shown in~\cite[Proposition~2.16]{cyclic}.
A modular algebra $B:\Pi |B \Envint\cat{O}|\to \End_\kappa^X$ over $\Pi |B \Envint\cat{O}|$
is a modular $\Envint \cat{O}$-algebra that sends all morphisms in the categories of operations to isomorphisms.
Via the functors
\begin{align}
\cat{O}(T)\to(\Envint \cat{O})(T) \ra{B_T}\End_\kappa^X(T) \ , 
\end{align}
we obtain a cyclic $\cat{O}$-algebra $RB:\cat{O}\to\End_\kappa^X$ that we call the \emph{restriction} of $B$.
The definition of 1-morphisms and 2-morphisms of cyclic and modular algebras entails once again that this assignment is in fact functorial, i.e.\ we obtain a functor
\begin{align}
R:\ModAlg(\Pi |B \Envint\cat{O}|)\to \CycAlg (\cat{O}  )
\end{align}
between 2-groupoids.

\begin{remark}
	If $\cat{O}$ admits an operadic identity $1_{\cat{O}} \in \cat{O}(T_1)$, then the modular envelope $\Envint\cat{O}$ also admits an operadic identity given by $(\id_{T_1},1_{\cat{O}})\in \Envint\mathcal{O}(T_1)$.
	Furthermore, this is also an operadic identity for $\Pi |B\Envint\cat{O}|$. When evaluating the modular extension
	$\widehat{A}$ of a cyclic $\mathcal{O}$-algebra $A$ on the operadic identity $(\id_{T_1},1_{\cat{O}})$, we find 
	\begin{align}
	\widehat{A}_{T_1}(\id_{T_1},1_{\cat{O}}) = \End^X_\kappa (\id_{T_1})(A_{T_1}(1_{\cat{O}})) \cong 
	A_{T_1}(1_{\cat{O}}) \cong \kappa \ , 
	\end{align} 
	where the last isomorphism comes from the condition that $A$ is compatible with the operadic identity.
	Conversely, if $B$ is a modular $\Pi |BU\cat{O}|$-algebra compatible with the identity by virtue of an isomorphism $B_{T_1}(\id_{T_1},1_{\cat{O}})\cong \kappa $, we find for the restriction  $RB$
	\begin{align}
	RB_{T_1}(1_{\cat{O}})= B_{T_1}(\id_{T_1},1_{\cat{O}})\cong \kappa \  \ .
	\end{align} 
	In summary, the modular extension and restriction
	defined above are compatible with our conventions regarding operadic identities.
\end{remark}

\begin{theorem}\label{thmositionmodext}
	Let $\cat{O}$ be a groupoid-valued cyclic operad and $\cat{M}$ a symmetric monoidal bicategory.
	The modular extension of cyclic algebras and the  restriction afford mutually weakly inverse equivalences between the 2-groupoids of
	cyclic $\cat{O}$-algebras in $\cat{M}$
	and  modular $\Pi |B \Envint\cat{O}|$-algebras in $\cat{M}$, i.e.\
	\begin{equation}
	\begin{tikzcd}
	\CycAlg(\cat{O}) \ar[rrrr, shift left=2,"\text{modular extension}"] &&\simeq&& \ar[llll, shift left=2,"\text{restriction}"] \ModAlg( \Pi |B \Envint\cat{O}|) \ . 
	\end{tikzcd}
	\end{equation}
	
\end{theorem}

\begin{proof}
	\begin{pnum}
		
		\item Let $A:\cat{O}\to\End_\kappa^X$ be a cyclic $\cat{O}$-algebra and $\Gamma:T\to T'$ a morphism in $\Forests$. 
		By definition the evaluation of $R\widehat{A}$ on $\Gamma$ is the natural isomorphism
		\begin{equation}\label{eqnthisdiagramhere}\begin{tikzcd}
		\mathcal{O}(T) \ar[dd,swap,"\mathcal{O}(\Gamma)"]  \ar[]{rr}{o \mapsto (\id_T,o) } & & (\Envint\cat{O})(T) \ar[dd,"(\Envint\cat{O})(\Gamma)"] \ar[lldd, Rightarrow]   \ar[]{rr}{\widehat A_T }&&  \End_\kappa^X(T)  \ar{dd}{     \End_\kappa^X(\Gamma)  }  \ar[lldd, Rightarrow,"\widehat A_\Gamma"]  \\
		& & \\
		\mathcal{O}(T') \ar[swap]{rr}{o \mapsto (\id_{T'},o)} & &   (\Envint\cat{O})(T') \ar[]{rr}{\widehat A_{T'} } & &  \End_\kappa^X(T') \ , 
		\end{tikzcd}  
		\end{equation}
		where the first square is filled by the natural transformation which at $o\in \cat{O}(T)$
		is given by the morphism $(\Gamma,o)\to (\id_{T'},   \cat{O}(\Gamma)o)$ in $(\Envint\cat{O})(T')$ induced by $\Gamma$. Note that $\widehat{A}_T$ and $\widehat{A}_{T'}$ send all morphisms to isomorphisms, so the composition of the two natural transformations is really a natural isomorphism.
		From the definition of $\widehat{A}$, it follows that the upper and lower horizontal composition in~\eqref{eqnthisdiagramhere}
		agrees with $A_T$ and $A_{T'}$, respectively.
		Moreover, the total natural isomorphism evaluated at $o\in \cat{O}(T)$ is given by the evaluation of $\widehat{A}_{T'}$ on the morphism
		$(\Gamma,o)\to (\id_{T'},   \cat{O}(\Gamma)o)$ in $(\Envint\cat{O})(T')$. 
		This evaluation is
		$\widehat{A}_{T'} (\Gamma,o)=\End_\kappa^X(\Gamma)A_o \to A_{\cat{O}(\Gamma)o}=\widehat{A}_{T'} (\id_{T'} ,    \cat{O}(\Gamma)o)$, i.e.\ the $o$-component of the natural isomorphism
		\begin{equation}
		\begin{tikzcd}
		\mathcal{O}(T) \ar[dd,swap,"\mathcal{O}(\Gamma)"]  \ar[]{rr}{A_T } & &   \End_\kappa^X(T)  \ar{dd}{     \End_\kappa^X(\Gamma)  }  \ar[lldd, Rightarrow,"A_\Gamma"]  \\
		& & \\
		\mathcal{O}(T') \ar[swap]{rr}{A_{T'}} & &  \End_\kappa^X(T') \ . 
		\end{tikzcd} 
		\end{equation}
		Hence, we may canonically identify $R\widehat{A}\simeq A$. 
		This is clearly natural in $A$ such that we have $R\, \widehat{-} \simeq \id_{\CycAlg(\cat{O})}$ as functors between 2-groupoids. This is a consequence of the fact that $\widehat{-}$ maps 1-morphisms and 2-morphisms to the same underlying 1-morphism or 2-morphism, respectively, in $\cat{M}$ and only extends the structure 2-isomorphisms. 
		
		\item
		Let $B$ be a modular $\Pi |B \Envint\cat{O}|$-algebra.
		We can describe $B$ equivalently as a modular $\Envint \cat{O}$-algebra
		$B:\Envint \cat{O}\to \End_\kappa^X$ sending all morphisms in the categories of operations to isomorphisms.
		Let $T\in\Graphs$ and $(\widetilde T,\Omega)\in \ell /T$ be given.
		The functor
		$
		\widehat{RB}_T :\Envint\cat{O}(T)\to\End_\kappa^X(T)
		$
		is determined by the fact that after precomposition with the functor $\cat{O}(\widetilde T)$ sending $o\in\cat{O}(\widetilde T)$ to $(\widetilde T,\Omega,o)\in(\Envint \cat{O})(T)$, it is given by
		\begin{align}
		\cat{O}(\widetilde T) \ra{(RB)_{\widetilde T}} \End_\kappa^X(\widetilde T) \ra{\End_\kappa^X(\Omega)} \End_\kappa^X(T) \ . \label{eqnwidehatR1}
		\end{align}
		By definition of $RB$ this functor agrees with the functor
		\begin{align}
		\cat{O}(\widetilde T) \ra{o \mapsto (\id_{\widetilde T} , o)} (\Envint \cat{O})(\widetilde T) \ra{B_{\widetilde T}} \End_\kappa^X(\widetilde T) \ra{\End_\kappa^X(\Omega)} \End_\kappa^X(T) \ . 
		\end{align}
		Via the natural isomorphism $\End_\kappa^X(\Omega)\circ B_{\widetilde T}\cong B_T \circ (\Envint \cat{O})(\Omega)$ that is part of the structure of $B$, 
		the functor~\eqref{eqnwidehatR1}
		is naturally isomorphic to
		\begin{align}
		\cat{O}(\widetilde T) \ra{o \mapsto (\id_{\widetilde T} , o)} (\Envint \cat{O})(\widetilde T) \ra{(\Envint \cat{O})(\Omega)}    (\Envint\cat{O})(T) \ra{B_T} \End_\kappa^X(T) \ . 
		\end{align}
		This gives us canonical natural isomorphisms $\widehat{RB}_T\stackrel{\alpha_T}{\cong} B_T$. 
		For any morphism $\Gamma:T\to T'$ in $\Graphs$,
		the natural isomorphism of functors
		\begin{equation}
		\begin{tikzcd}
		(\Envint	\mathcal{O})(T) \ar[dd,swap,"(\Envint \mathcal{O})(\Gamma)"]  \ar[rr, bend left =90, ""{name=S}, "B_T"]  \ar[""{name=U}]{rr}{\widehat{RB}_T }     \ar[from=S, to=U, shorten <=8, shorten >=8, Rightarrow,"\alpha_T"] & &   \End_\kappa^X(T)  \ar{dd}{     \End_\kappa^X(\Gamma)  }  \ar[lldd, Rightarrow,"\widehat{RB}_\Gamma"]  \\
		& & \\
		(\Envint	\mathcal{O})(T') \ar[swap,""{name=W}]{rr}{\widehat{RB}_{T'}} \ar[rr, bend right =90, ""{name=Q}, "B_{T'}"]        \ar[from=W, to=Q, shorten <=8, shorten >=8, Rightarrow,"\alpha_{T'}"] & &  \End_\kappa^X(T') \ . 
		\end{tikzcd} 
		\end{equation}
		agrees with $B_\Gamma$, i.e.\ we can canonically identify $\widehat{RB}\simeq B$. 
		Again, this is natural in $B$ such that $\widehat{R-}\simeq \id_{\ModAlg(\Envint \cat{O})}$ as functors between 2-groupoids. 
		
	\end{pnum}
\end{proof}

\section{The classification of ansular functors\label{secansular}}
In this section, we combine the Theorems~\ref{thmGenus1envelope} and~\ref{thmositionmodext}
with the main result of \cite{giansiracusa} to classify ansular functors.
Let us first state the main result of \cite{giansiracusa} in a slightly different form.

\begin{theorem}[unital version of $\text{\cite[Theorem~A]{giansiracusa}}$]\label{thmgengiansiracusa}
	There is a canonical map $\omega \colon \Envint \framed \to \Hbdy$ of category-valued modular operads 
	such that the induced map $\mathbb{L}\Env \framed \to |B \Hbdy|$ of topological modular operads is a bijection on $\pi_0$ and 
	 a homotopy equivalence on all path components except the one for the solid closed torus.
	\end{theorem}

\begin{proof} 
	The proof amounts to a rather technical combination of the results of~\cite{giansiracusa} with our different treatment of operations of arity $-1$ and zero. Let us give the details:
	Recall that we can describe $\framed$ as the genus zero restriction $\Hbdy_0$ of $\Hbdy$.
	For a corolla $T$,
an object of $\Envint \Hbdy_0 (T)$ is a 
connected graph $\Gamma$ with an identification of its legs with those of $T$ and elements
$H^{(j)}\in \Hbdy_0(T^{(j)}) $ for all corollas $T^{(j)}$ arising from $\Gamma$ by 
cutting open all internal edges. 
The functor $\omega \colon \Envint \Hbdy_0 (T) \to  \Hbdy (T)$ is defined by sending this 
data to the handlebody constructed by gluing together 
the genus zero handlebodies 
$H^{(j)}$ according to the graph $\Gamma$. 

Giansiracusa considers the 
restriction $\omega^\text{a}$ of this functor to the full subcategory of objects
whose underlying graph does not contain any vertices of valence one (the reason for this is that in \cite{giansiracusa} a non-unital version of the
framed $E_2$-operad is considered).  
Neither $\omega$ nor $\omega^\text{a}$ are arity-wise
equivalences of categories. However, by~\cite[Theorem A]{giansiracusa} $\omega^\text{a}$ induces, after applying nerve and geometric realization,
 a bijection on $\pi_0$ and  a homotopy 
equivalence on all path components except the one for the solid closed 
torus, the closed three-dimensional ball, and the three-dimensional ball with one embedded disk
(actually, the latter two cases
are excluded from the discussion altogether in~\cite{giansiracusa}).

One can easily observe that $\omega$ induces a bijection between the sets of path components, simply because this is the case for $\omega^\text{a}$.
Therefore, it remains to prove that $\omega$ induces a homotopy equivalence after taking nerve and geometric realization on all path components except the solid closed torus:

\begin{itemize}
	\item
 It is clear that $\omega$ induces a homotopy equivalence
on the components corresponding to the three-dimensional ball with and without an
embedded disk because the corresponding components
in $\Envint \framed$ and $\Hbdy$ are contractible.

 \item We now treat the path components which are not a three-dimensional ball with or without an embedded disk (because these have already been dealt with) or the solid closed torus (because this case is excluded in the statement of the result). 
  In these remaining cases,
 we use
 Proposition~\ref{Prop:no0vertices} to conclude that
 it is enough to show
 that $\omega$ restricted to 
 the category $\int \left( \catf{Gr}_{\operatorname{conn}}^{\operatorname{a},g}(T)\to\Forests \ra{ \framed}\Cat\right)$
 induces a homotopy equivalence except for the cases of $(T,g)$ that we have excluded or have already dealt with:
  the three-dimensional ball $(\bullet=T_{-1},0)$, solid closed torus $(\bullet,1)$, and three-dimensional ball with one embedded disk $(T_0,0)$.
 But having made these exclusions, the graph part of the objects in $\int \left( \catf{Gr}_{\operatorname{conn}}^{\operatorname{a},g}(T)\to\Forests \ra{ \framed}\Cat\right)$
 does not contain vertices of valence one by definition. Hence, the restriction of $\omega$ to this
 subcategory agrees with $\omega^{\operatorname{a}}$, and the result follows from Giansiracusa's \cite[Theorem A]{giansiracusa}.
 \end{itemize}

	\end{proof}
	
	\begin{theorem}\label{thmmodenvpi}
		The canonical map $\omega \colon \Envint \framed \to \Hbdy$ from Theorem~\ref{thmgengiansiracusa}
		induces an equivalence $\Pi |B\Envint \framed| \ra{\simeq} \Hbdy$ 
		 of groupoid-valued modular operads.
	\end{theorem}

\begin{proof}
	Thanks to $|B \Envint \framed|\simeq \mathbb{L}\Env \framed$ (see~\eqref{eqnenvenvint}),
	this follows from Theorem~\ref{thmgengiansiracusa} and
	the fact that, on the component of the solid closed torus, the map
	$\omega: \mathbb{L}\Env \framed \to\Hbdy$
	is given, up to homotopy equivalence, by the projection map \begin{align} B\Diff(H_{1,0})\to B\Map(H_{1,0}) \ , \label{eqnprojectionmap}
	\end{align} after using the homotopy equivalence
	$\envframedt   \simeq   B \Diff (H_{1,0})$ afforded by Theorem~\ref{thmGenus1envelope}.
	Once we show this, the proof is done because the map \eqref{eqnprojectionmap} induces an equivalence after taking  fundamental groupoids.
	
	To see that \eqref{eqnprojectionmap} is really the map 
that we extract from $\omega$, we observe that 
	the map $\omega$, as constructed in the proof of 
	Theorem~\ref{thmgengiansiracusa},
	 sends the twist $\theta$, seen as an automorphism of the object 
	 
	 \small
	  \begin{align} \left(  \text{circle with one vertex} , \text{solid cylinder} \in \framed(T_1)  \right) \in \int \left( \catf{Gr}_{\operatorname{conn}}^{\operatorname{a},1}(\bullet)\to\Forests \ra{ \framed}\Cat\right)\ , \end{align} \normalsize to the Dehn twist of the solid torus. Furthermore, the automorphism
	of 
	
	\small
	 \begin{align} \left(  \text{circle with two vertices} , \text{two solid cylinders} \in \framed(T_1 \sqcup T_1)  \right) \in \int \left( \catf{Gr}_{\operatorname{conn}}^{\operatorname{a},1}(\bullet)\to\Forests \ra{ \framed}\Cat\right)\ , \end{align} \normalsize
	  which exchanges the vertices is sent to the rotation
	$R$ from~\eqref{eqnTandRmatrix}. 
	This implies that the map extracted from $\omega$ at genus zero is really~\eqref{eqnprojectionmap} and therefore finishes the proof.
	\end{proof}

\begin{theorem}\label{thmansular}
	For any symmetric monoidal bicategory $\cat{M}$,
	the modular extension and the genus zero restriction 
	afford mutually weakly inverse equivalences between the 2-groupoids of
	$\cat{M}$-valued ansular functors 
	and cyclic framed $E_2$-algebras in $\cat{M}$.
	\end{theorem}

\begin{proof}
	Theorem~\ref{thmositionmodext}, when applied to $\framed$, gives us
	an equivalence of 2-groupoids
	\begin{align}\CycAlg(\framed) \simeq \ModAlg(\Pi |B\Envint \framed|)       \label{eqnclassrhs}
	\end{align} via modular extension and genus zero restriction.
	By Theorem~\ref{thmmodenvpi} we have an canonical equivalence $\Pi |B\Envint \framed|\simeq \Hbdy$ of modular operads. 
	Thanks to the homotopy invariance of 2-groupoids of modular algebras \cite[Theorem~2.18]{cyclic}, this implies that the 2-groupoid of modular $\Pi |B\Envint \framed|$-algebras (i.e.\ the right hand side of~\eqref{eqnclassrhs}) is equivalent to the 2-groupoid of modular $\Hbdy$-algebras, i.e.\ ansular functors.
	\end{proof}

In order to arrive at a classification of ansular functors, we need to describe cyclic framed $E_2$-algebras. This part has already been accomplished in \cite{cyclic}. Let us give a brief summary:
A \emph{balanced braided algebra} in a symmetric monoidal bicategory $\cat{M}$ is an object $A \in \cat{M}$
equipped with the following structure: \begin{itemize}
	\item A multiplication $\mu : A \otimes A \to A$ which is associative and unital
	(the unit is a 1-morphism $u:I\to A$ from the unit $I$ of $\cat{M}$ to $A$)
	 up to coherent isomorphism.
	\item An isomorphism $c: \mu \to \mu^\opp = \mu \circ \tau$ (here $\tau$ is the symmetric braiding of $\cat{M}$) called \emph{braiding} subject to the usual conditions known from the definition of a braided monoidal category.
	\item An isomorphism $\theta : \id_A \to \id_A$ called \emph{balancing} subject to the requirements
	\begin{align}
		\theta \circ \mu  &= c^2 \circ \mu(\theta  \otimes \theta) \ , \\
		\theta \circ u &= \id_u \ . 
		\end{align}
	\end{itemize}

It is well-known that balanced braided algebras are exactly \emph{non-cyclic} framed $E_2$-algebras:

\begin{theorem}[Wahl \cite{wahlthesis}, Salvatore-Wahl \cite{salvatorewahl}]\label{thmsawa}
	Framed $E_2$-algebras in a symmetric monoid\-al bicategory $\cat{M}$
	are equivalent to balanced braided algebras in $\cat{M}$.
	\end{theorem}

It is proven in \cite[Theorem~5.4]{cyclic}
that the structure needed on a framed $E_2$-algebra in order to make it a \emph{cyclic} framed $E_2$-algebra is exactly the following:

\begin{definition}\label{defselfdual}
	A \emph{self-dual balanced braided algebra}
	$A$ in a symmetric monoidal bicategory $\cat{M}$ is a balanced braided algebra with product $\mu$, unit $u$, balancing $\theta$ and, additionally, a non-degenerate pairing 
	$\kappa : A \otimes A \to I$, i.e.\ a 1-morphism which exhibits $A$ as its own dual in the homotopy category of $\cat{M}$, and an isomorphism $\gamma : \kappa(u\otimes \mu)\to \kappa$ such that  \begin{itemize}\item the isomorphism
		$\kappa(u\otimes \mu(u,-))\longrightarrow \kappa(u,-)$ induced by 
		$\gamma$ agrees with the isomorphism induced by the unit constraint,
		\item $\kappa(\theta \otimes \id_A)=\kappa(\id_A \otimes \theta)$.
	\end{itemize}
	\end{definition}

\begin{theorem}[$\text{\cite[Theorem~5.4]{cyclic}}$]
	Cyclic framed $E_2$-algebras in a symmetric monoidal bicategory $\cat{M}$
	are equivalent to self-dual balanced braided algebras in $\cat{M}$.
\end{theorem}

By combining this result with Theorem~\ref{thmansular}, we finally arrive at:

\begin{theorem}[Classification of ansular functors]\label{thmansular2}
	For any symmetric monoidal bicategory $\cat{M}$,
	the modular extension and the genus zero restriction 
	afford mutually weakly inverse equivalences between the bicategories of
	$\cat{M}$-valued ansular functors 
	and self-dual
	balanced braided algebras in $\cat{M}$.
\end{theorem}

As a consequence of Theorem~\ref{thmansular2}, we can deduce that for a cyclic framed $E_2$-algebra $A$, the symmetries of $A$ act on the ansular functor $\widehat{A}$ obtained by modular extension; in particular, they intertwine with the handlebody group representations. In fact, a much stronger statement is true:
\begin{corollary}
	For any cyclic framed $E_2$-algebra $A$ in $\cat{M}$, the modular extension provides an equi\-valence
	\begin{align} \Aut(A) \ra{\simeq} \Aut\left(\widehat{A}\right) 
		\end{align}
	between the 2-group of autoequivalences of $A$, as cyclic framed $E_2$-algebra,
	 to the 2-group of autoequivalences of the ansular functor $\widehat{A}$.
	\end{corollary}

In the case $\cat{M}=\Lexf$, it is proven in \cite[Theorem~5.11]{cyclic} that
cyclic framed $E_2$-algebras are 	equivalent to ribbon Grothendieck-Verdier categories in the sense of Boyarchenko-Drinfeld \cite{bd}. 	A \emph{Grothendieck-Verdier category} in $\Lexf$ is a monoidal category $\cat{C}$ in $\Lexf$ with a distinguished object $K\in\cat{C}$, the \emph{dualizing object}, such that for all $X\in\cat{C}$ the functor $\cat{C}(K,X\otimes-)$ is representable in a way that the functor $D:\cat{C}\to \cat{C}^\opp$ (called \emph{duality functor}) sending $X$ to a representing object $DX$ is an equivalence. A \emph{ribbon Grothendieck-Verdier category} in $\Lexf$ is a Grothendieck-Verdier category in $\Lexf$ together with a braiding, i.e.\ coherent natural isomorphisms $X\otimes Y\cong Y\otimes X$ for $X,Y\in\cat{C}$, and a balancing, i.e.\ a natural automorphism of $\id_\cat{C}$ whose components $\theta_X:X\to X$ satisfy $\theta_I=\id_I$ and $\theta_{X\otimes Y}=c_{Y,X}c_{X,Y}(\theta_X\otimes\theta_Y)$; additionally, one requires $\theta_{DX}=D\theta_X$. 
By replacing left exact functors with right exact ones, one obtains the symmetric monoidal bicategory $\Rexf$ instead of $\Lexf$. When defining (ribbon) Grothendieck-Verdier categories in $\Rexf$, one needs to ask for the representability of $\cat{C}(X\otimes-,K)$ rather than $\cat{C}(K,X\otimes-)$. 
In fact, this is the convention for Grothendieck-Verdier structures that is used in \cite{bd}; the one used above for $\Lexf$ is obtained via dualization.
With these slightly different conventions for Grothendieck-Verdier duality in $\Lexf$ and $\Rexf$,
we obtain from
\cite[Theorem~5.11]{cyclic}
the following `linear version' of
Theorem~\ref{thmansular}:

\begin{theorem}[Classification of ansular functors --- linear version]\label{thmclassansfunctor}
	The modular extension and the genus zero restriction afford mutually weakly inverse equivalences between
ansular functors with values in  $\Lexf$ or $\Rexf$
	and ribbon Grothendieck-Verdier categories in $\Lexf$ or $\Rexf$. 
\end{theorem}

\begin{remark}\label{remgenuszero}
	Let $\cat{C}$ be a ribbon Grothendieck-Verdier category in $\Lexf$.
	Then the genus zero restriction of the ansular functor associated to $\cat{C}$ is, of course, $\cat{C}$ itself. 
	This means that the functor $\cat{C}^{\boxtimes n} \to \cat{C}$ associated to a genus zero handlebody with $n+1$ embedded disks (to be thought of as $n$ inputs and one output) is
	given by the $n$-fold monoidal product
	\begin{align}\cat{C}^{\boxtimes n}\to\cat{C} \ , \quad X_1 \boxtimes \dots \boxtimes X_n \mapsto X_1 \otimes \dots \otimes X_n  \ . \label{eqnCn}
		\end{align}
	We are choosing here an ordering of the embedded disks, but this could be avoided by using unordered monoidal products.
	Using the cyclic structure, we may understand~\eqref{eqnCn} as an operation with $n+1$ inputs and zero outputs. This gives us the functor
	\small
	\begin{align}
		\cat{C}^{\boxtimes (n+1)}\to\fVect \ , \quad  X_0 \boxtimes \dots \boxtimes X_n \mapsto \cat{C}(DX_0,X_1\otimes\dots\otimes X_n)\cong \cat{C}(K,X_0\otimes\dots \otimes X_n) \ .  \label{eqnCn2}
		\end{align} \normalsize
	Indeed, by the definition of the cyclic structure the functor $\cat{C}^{\boxtimes (n+1)}\to\fVect$
	corresponding to~\eqref{eqnCn} sends $ X_0 \boxtimes \dots \boxtimes X_n$ to $\kappa(X_0,X_1\otimes\dots\otimes X_n)$, where $\kappa :\cat{C}\boxtimes\cat{C}\to\fVect$ is the symmetric non-degenerate pairing that is part of the cyclic structure. Now the statement follows from the canonical isomorphism $\kappa(X,Y)\cong \cat{C}(DX,Y)$ for all $X,Y\in\cat{C}$ \cite[Lemma~2.21]{cyclic}.
	By virtue of the braiding and the balancing we have actions of the mapping class groups of genus zero handlebodies (equivalently: mapping class groups of genus zero surfaces), namely ribbon braid groups, on the functors~\eqref{eqnCn2}
	(this is a consequence Theorem~\ref{thmsawa}).
	Theorem~\ref{thmclassansfunctor} now tells us that there is a unique extension of~\eqref{eqnCn2} to an ansular functor. 
	\end{remark}

\spaceplease
\section{Applications}
The classification of ansular functors allows us to write down \emph{any} ansular functor in $\Lexf$ very explicitly, simply because we know thanks to the classification result that all ansular functors in $\Lexf$ are obtained by the modular extension procedure.
This will also make heavy use of the computations in \cite[Section 7.2]{cyclic}. These, however, are a priori only valid for handlebodies which are not the solid closed torus.
For this reason, we first need to treat the solid torus separately. 

A central ingredient will be the canonical coend $\mathbb{F} \in \cat{C}$ of a Grothendieck-Verdier category $\cat{C}$ in $\Lexf$ which is defined as the image of the coend
$\int^{X\in\cat{C}} DX\boxtimes X\in\cat{C}\boxtimes\cat{C}$ under the monoidal product. 
This generalizes the Lyubashenko coend \cite{lyu,lyulex} of a finite tensor category. For this reason, we use the notation $\mathbb{F}=\int^{X\in\cat{C}} DX \otimes X$, but we want to emphasize that the proper definition is really $\mathbb{F} := \otimes \left(     \int^{X\in\cat{C}} DX\boxtimes X  \right)$. 

\begin{lemma}\label{lemmareptorus}
	Let $\cat{C}$ be ribbon Grothendieck-Verdier category $\cat{C}$ in $\Lexf$ with dualizing object $K$ and duality functor $D$.
The $\Map(H_{1,0})$-representation that the ansular functor associated to $\cat{C}$ gives rise to has the underlying vector space 
$\cat{C}(K,\mathbb{F})$ and
can be explicitly described as follows:
\begin{pnum}

\item The mapping class group element $T$ (Dehn twist along the waist of the solid closed torus, see~\eqref{eqnTandRmatrix})
acts by the automorphism of $\cat{C}(K,\mathbb{F})$ that is induced by the automorphism
\begin{align} t : \mathbb{F}=\int^{X\in\cat{C}}X\otimes DX \ra{\theta_X\otimes DX} \int^{X\in\cat{C}}X\otimes DX=\mathbb{F} \ , \label{eqnautT}
\end{align}
where $\theta_X :X \to X$ is the balancing.\label{pnumT}

\item\label{pnumR}
The mapping class group element $R$
(rotation by $\pi$, see~\eqref{eqnTandRmatrix})
 acts by the automorphism  
of $\cat{C}(K,\mathbb{F})$ that is induced by the automorphism
\begin{align}
r : \mathbb{F}=\int^{X\in\cat{C}}X\otimes DX \ra{  (\theta_{DX}\otimes X)\circ c_{X,DX}   }\int^{X\in\cat{C}}DX \otimes X \cong \int^{X\in\cat{C}}X\otimes DX=\mathbb{F} \ . \label{eqnautR}
\end{align}

\end{pnum}

\end{lemma}

\begin{proof} Let $X\in\cat{C}$.
	The vector space associated to a three-dimensional ball with two embedded disks (a solid cylinder) labeled by $X$ and $DX$ (both seen as inputs)
	is $\cat{C}(K,X\otimes DX)$, see Remark~\ref{remgenuszero}.
	We can obtain the solid closed torus by gluing the ends of a solid cylinder together.
	On the level of the ansular functor, this gluing translates to a Lyubashenko's left exact coend $\oint$; this is the 
	excision property of the modular extension of a cyclic algebra~\cite[Theorem~6.4]{cyclic}.
	Therefore, the ansular functor associated to $\cat{C}$ assigns to the solid closed torus the vector space
	$\oint^{X\in\cat{C}} \cat{C}(K,X\otimes DX)$. This is canonically isomorphic to $\cat{C}(K,\mathbb{F})$
	by \cite[Lemma~2.25]{cyclic}.

	Statement~\ref{pnumT} follows from the fact that,
	when writing the solid closed torus as the result of gluing the ends of a cylinder together, the Dehn twist
	amounts to a rotation of one of these disks and hence, on the algebraic level, acts by the balancing $\theta$ (actually, it does not matter which one is rotated; algebraically, this is reflected by the fact that the maps $\theta_X \otimes DX$ and $X \otimes \theta_{DX}=X\otimes D\theta_X$ induce the same map when passing to the coend).
	 
	For the proof of~\ref{pnumR}, we consider again the effect of $R$ on the solid cylinder with identified ends. The rotation exchanges the two embedded disks (i.e.\ we have to apply a braiding), and each of the disks experiences a half rotation. Relative to the second disk, this amounts to a full rotation of the first disk to which we therefore have to apply the balancing. 
	\end{proof}

\begin{remark}
	It is easy and very instructive to check algebraically by hand that through~\ref{pnumT} and~\ref{pnumR} in Lemma~\ref{lemmareptorus}
	we obtain indeed a $\mathbb{Z}\times\mathbb{Z}_2$-action: The automorphisms~\eqref{eqnautT} and~\eqref{eqnautR} 
commute by the naturality of the twist 
and the universal property of the coend
and hence give a $\mathbb{Z}\times\mathbb{Z}$-action on $\mathbb{F}$. 
Since $r^2=\theta_\mathbb{F}$, as automorphism of $\mathbb{F}$, and $\theta_K=\id_K$, the induced action of $R^2$ on $\cat{C}(K,\mathbb{F})$ is trivial.
As a consequence, we obtain a $\mathbb{Z}\times\mathbb{Z}_2$-action on $\cat{C}(K,\mathbb{F})$.
\end{remark}

We can now easily describe \emph{all} ansular functors with values in $\Lexf$ explicitly:

\begin{corollary}\label{corallansular}
	Given an arbitrary ansular functor with values in $\Lexf$, let $\cat{C}\in\Lexf$ be its genus zero part, i.e.\ a ribbon Grothendieck-Verdier category. Then the value of the ansular functor on a handlebody $H_{g,n}$ of genus $g$ and $n$ embedded disks labeled by $X_1,\dots,X_n \in \cat{C}$ (we pick here an order for the embedded disks) is isomorphic to the morphism space
	\begin{align}\cat{C}(K,X_1\otimes\dots \otimes X_n\otimes \mathbb{F}^{\otimes g}) \ .         \label{eqnvspF}
	\end{align} 
	\end{corollary}

\begin{proof}
	Thanks to Theorem~\ref{thmclassansfunctor},
	we know that we can restrict any given ansular functor to its genus zero part and recover it, up to canonical equivalence,
	by modular extension.
	The modular extension of a cyclic algebra can be computed via excision \cite[Theorem~6.4]{cyclic}, see \cite[Section 7.2]{cyclic} for an explicit description away from the component of the solid closed torus and Lemma~\ref{lemmareptorus} above for the remaining case
	of the solid closed torus.
	\end{proof}

\begin{example}\label{exhopf}
	Any finite ribbon category in the sense of \cite{egno} is a ribbon Grothendieck-Verdier category.
	In particular, the category of finite-dimensional modules over a finite-dimensional ribbon Hopf algebra $H$ is a ribbon Grothendieck-Verdier category.
	In this case, the vector space attached to a handlebody with genus $g$ is isomorphic to $\Hom_H(k, X_1\otimes\dots \otimes X_n\otimes (H_\text{coadj}^* )^{\otimes g})$, where
	$H_\text{coadj}^*$ is the dual of $H$ with the coadjoint action. 
	This example is discussed in \cite[Corollary~7.13]{cyclic}, see also~\cite[Example~3.8]{mwdiff}. Therefore, we omit the details here.
	\end{example}

\begin{corollary}\label{corrcat}
		The genus zero restriction 
		and modular extension afford mutually weakly inverse equivalences \begin{align}
	\left\{	\begin{array}{c} 
	\text{ansular functors $A$ with values in  $\Lexf$} \\ \text{such that $A(H_{0,1} : \emptyset \to \mathbb{D}^2)$} \\ \text{is left adjoint to $A(H_{1,0} : \mathbb{D}^2 \to \emptyset)$}
	\end{array}\right\} \simeq \left\{  \begin{array}{c} \text{ribbon Grothendieck-Verdier} \\ \text{categories in $\Lexf$} \\
	\text{such that $K\cong I$} \end{array}
	 \right\} \ . \end{align}
	\end{corollary} 
The notation 
$A(H_{0,1} : \emptyset \to \mathbb{D}^2)$ and
$A(H_{1,0} : \mathbb{D}^2 \to \emptyset)$ was introduced in
 Remark~\ref{rem: concrete description ansular functor}.

The adjointness relation on the left hand side and the isomorphism $K\cong I$ 
on the right hand side
are considered as properties.
The 2-groupoids on both sides are the full sub-2-groupoids of the 2-groupoids
of ansular functors and ribbon Grothendieck-Verdier categories
spanned by the objects with the respective properties.  
The mathematical structure on the right hand side is closely related to the (ribbon version of the) notion of an
\emph{$r$-category} \cite[Definition~0.5]{bd}, a monoidal category whose monoidal unit is dualizing. 
But we have to be careful: The property on the right hand side in Corollary~\ref{corrcat} 
is a property of a ribbon Grothendieck-Verdier category while being a $r$-category is a property of a monoidal category;
 and even if a monoidal category is an $r$-category, it may have a dualizing object that is not the monoidal unit~\cite[Section~1.2]{bd}.

\begin{proof}[\slshape Proof of Corollary~\ref{corrcat}] Suppose that $A$ has underlying category $\cat{C} \in\Lexf$.
	From the definition of the modular extension, it follows that the 
	 functor $A(H_{0,1} : \emptyset \to \mathbb{D}^2) : \fVect \to \cat{C}$ is given by the unit of $\cat{C}$ (note that a left exact functor $\fVect \to \cat{C}$ is determined by its value on the ground field and hence amounts to an object in $\cat{C}$) while the functor $A(H_{1,0} : \mathbb{D}^2 \to \emptyset):\cat{C}\to\fVect$ is given by $\cat{C}(K,-)$. 
	 Therefore, $A(H_{0,1}: \emptyset \to \mathbb{D}^2)\dashv A(H_{1,0} : \mathbb{D}^2 \to \emptyset) $ holds if and only if $I\dashv \cat{C}(K,-)$, where `$\dashv$' means `is left adjoint to'. The adjunction
	 $I\dashv \cat{C}(K,-)$ means exactly that we have isomorphisms
	 $\cat{C}(I, X)\cong \Hom_k(k,\cat{C}(K,X))\cong\cat{C}(K,X)$ natural in $X$ which, by the Yoneda Lemma, is exactly an isomorphism $I\cong K$.
	 Now the statement follows from Theorem~\ref{thmclassansfunctor}.
	\end{proof}

\begin{corollary}
	For any ansular functor with values in $\Lexf$ with genus zero restriction $\cat{C}\in\Lexf$, the vector space associated to the solid closed torus
	(without its handlebody group action)
	depends only on $\cat{C}$ as linear category. 
	The vector space associated to the handlebody of genus $g$ without embedded disks for $g\neq 1$ generally depends on the ribbon Grothendieck-Verdier structure of $\cat{C}$ and not only on $\cat{C}$ as linear category. 
	\end{corollary}

\begin{proof}
	The vector space associated to the solid closed torus is $\cat{C}(K,\mathbb{F})\cong  \oint^{X\in\cat{C}} \cat{C}(K,DX\otimes X)\cong \oint^{X\in\cat{C}} \cat{C}(X,X)$, see Lemma~\ref{lemmareptorus} and its proof. This depends only on $\cat{C}$ as linear category.
	
	Now we need to show that for a handlebody $H_{g}$ of genus $g\neq 1$ and no embedded disks, the vector space generally depends on the ribbon 
	Grothendieck-Verdier struture. To this end, we use the ribbon Grothendieck-Verdier category $\Vect_G^\omega$ associated to
	an abelian 3-cocycle on a finite abelian group $G$ with coefficients
	in $\mathbb{C}^\times$, see~\cite[Theorem~4.2.2]{zetsche} and additionally \cite[Example~5.14]{cyclic} for a summary. The underlying linear category is given by finite-dimensional $G$-graded complex vector spaces. In general, the abelian 3-cocycle is used to deform both the braiding and the associator of $\Vect_G$. For the proof, it is enough to consider $\omega=0$. The monoidal product of $G$-graded vector spaces $V=(V_a)_{a\in G}$ and $W=(W_a)_{a\in G}$ is given by $(V\otimes W)_{g}=\bigoplus_{bc=g} V_b \otimes W_c$. The duality is given by $D=\mathbb{C}_{a_0} \otimes (-)^*$, where $\mathbb{C}_{a_0}$ is $\mathbb{C}$ seen as $G$-graded vector space concentrated in degree $a_0 \in G$ with $a_0 = b_0^{-2}$ for some $b_0 \in G$. By \cite[Example~7.12]{cyclic}, the vector space  associated to the handlebody $H_g$ is
	$\mathbb{C}[G]^{\otimes g} \delta_{a_0,a_0^g}$. We can always choose $a_0=1$. Then this vector space is always non-zero. 
	But for a fixed $g$, we can by a suitable choice of $G$ and $a_0$ arrange $a_0^{g-1} \neq 1$. Then $\mathbb{C}[G]^{\otimes g} \delta_{a_0,a_0^g}=0$.
	In summary, we have exhibited an example in which the vector space associated to $H_g$ for $g \neq 1$ can be made non-zero or zero by a suitable choice the ribbon Grothendieck-Verdier structure, without changing the underlying linear category.
	\end{proof}
 
 \begin{example}
 	Let $A$ be an ansular functor in $\Lexf$ whose genus zero restriction $\cat{C}$ is a \emph{semisimple} ribbon Grothendieck-Verdier category with \emph{simple unit}.
 	Denote by $X_0=I,\dots,X_{n-1}$ a basis of simple objects and by
 	$N_{\alpha\beta}^\gamma \in \mathbb{Z}$ the
 	 fusion coefficients characterized by $X_\alpha \otimes X_\beta \cong \bigoplus_{\gamma=0}^{n-1} N_{\alpha\beta}^\gamma X_\gamma$. Denote by $\bar \alpha \in \{0,\dots,n-1\}$ the unique index such that $X_{\bar \alpha} \cong DX_\alpha$. 
 	The well-definedness of $\bar \alpha$ uses that $DX_\alpha$ is simple because $D$ is an equivalence. Since $\cat{C}$ is pivotal by \cite[Corollary~8.3]{bd}, $\bar{\bar\alpha}=\alpha$.
 	We can now express all dimensions $	\dim A (H_g) $
 	of the vector spaces associated to a handlebody of genus $g$ (for simplicity without embedded disks)
 	in terms of the fusion coefficients:
 	Corollary~\ref{corallansular} tells us $A(H_0) \cong \cat{C}(K,I)=\cat{C}(X_{\bar 0},X_0)$ and $A(H_1)\cong \cat{C}(K,\mathbb{F})$.
 	This proves 
 	\begin{align}
 	\dim A(H_0)=\delta_{0,\bar 0} \ . 
 	\end{align}
 	For $H_1$, we use $\mathbb{F} \cong \bigoplus_{\alpha=0}^{n-1} X_{\bar \alpha}\otimes X_\alpha$ from \cite[Corollary~5.1.9]{kl} and conclude 
 	\begin{align} \dim A(H_1)=n 
 	\end{align}
 	thanks to $\cat{C}(K,X_{\bar \alpha}\otimes X_\alpha)\cong \cat{C}(X_\alpha,X_\alpha)\cong k$.
 	For $g\ge 2$, we use again Corollary~\ref{corallansular} which tells us $A(H_g) = \cat{C}(K,\mathbb{F}^{\otimes g})\cong \cat{C}(D\mathbb{F},\mathbb{F}^{\otimes (g-1)})$. Since $\mathbb{F} \cong \bigoplus_{\alpha=0}^{n-1} X_{\bar \alpha}\otimes X_\alpha \cong \bigoplus_{\alpha,\beta=0}^{n-1}  N_{\bar \alpha \alpha}^\beta X_\beta $ and hence
 	$D\mathbb{F}\cong \bigoplus_{\alpha,\beta=0}^{n-1}  N_{\bar \alpha \alpha}^\beta X_{\bar \beta}$, we can explicitly write out the dimension of $A(H_g)$. 
 	 For $g=2$,
 	\begin{align}
 	\dim A(H_2)= \dim \cat{C} \left(    D\mathbb{F},\mathbb{F}   \right)= \sum_{\alpha,\beta,\alpha',\beta'=0}^{n-1} N_{\bar \alpha \alpha}^\beta N_{\bar \alpha' \alpha'}^{\beta'} \delta_{\beta \bar \beta'} = \sum_{\alpha,\beta,\alpha'=0}^{n-1}  N_{\bar \alpha \alpha}^\beta N_{\bar \alpha' \alpha'}^{\bar \beta} \ . 
 	\end{align}
 	 For $g>2$, we just iterate the procedure, but no new insight enters.
 	 These dimensions can be computed for all 
 	 ansular functors whose circle category is semisimple and has a simple unit.
 	 These dimension formulas also hold for semisimple modular functors based on a category with simple unit.
 	 This can indeed be verified for certain specific constructions such as the Reshetikhin-Turaev modular functor for a semisimple modular category \cite{rt1,turaev}. However, the point is that the considerations above are not limited to certain constructions that happen to be known, but for \emph{all} possible ones. 
 	\end{example}

For a vertex operator algebra $V$,
appropriate choices for a category 
of $V$-modules
yield a ribbon Grothendieck-Verdier category by~\cite[Theorem~2.12]{alsw}.
More precisely, the category of modules must contain $V$, must be
closed under taking
contragredients and must satisfy the two technical assumptions of \cite[Proposition~2.10]{alsw}.
The duality functor sends a module $M$ to its contragredient module $M^*$. 
	If $V$ is $C_2$-cofinite, this category will be finite and the monoidal product will be right exact, i.e.\ we obtain a ribbon Grothendieck-Verdier category in $\Rexf$. Consequently, the opposite category will naturally be a ribbon Grothendieck-Verdier category in $\Lexf$.
	We will denote this category by $\cat{C}_V$.
	
	\begin{corollary}\label{corvoa}
	Let $V$ be a $C_2$-cofinite vertex operator operator algebra.
	Then the ribbon Grothen\-dieck-Verdier structure in $\Lexf$ on any category $\cat{C}_V$ of $V$-modules 
	subject to the above-mentioned restrictions extends
	to an ansular functor.
	Up to equivalence, this extension is unique.
	The vertex operator algebra is self-dual, i.e.\ $V^*\cong V$, if and only if the functors obtained by evaluation of the ansular functor on oppositely oriented disks are adjoint.
\end{corollary}

\begin{proof}
	Thanks to \cite[Theorem~2.12]{alsw}, this is a consequence of Theorem~\ref{thmclassansfunctor}.
	For the last statement on self-duality, one needs to observe that $V$ is by definition self-dual if and only if the dualizing object of $\cat{C}_V$ isomorphic to the monoidal unit.
	Then the last statement follows from 
	Corollary~\ref{corrcat}.
	\end{proof}

	A description of the ansular functor associated to $V$ in terms of the category $\cat{C}_V$ of modules is given in Corollary~\ref{corallansular}
	 while a description in terms of the vertex operator algebra $V$ is beyond the scope of this article.
	The reason for our interest in vertex operator algebras is that they provide non-exact examples:

\begin{corollary}\label{cornonexact}	
	There exist $\Lexf$-valued ansular functors whose underlying monoidal category is not exact and hence not rigid.
	\end{corollary}

\begin{proof}
	In Corollary~\ref{corvoa}, we can specialize $V$ to be the $\cat{W}_{2,3}$ triplet model \cite{grw} which has a non-exact monoidal product.
	\end{proof}

\end{document}